\documentclass[12pt]{amsbook}
\usepackage[all,cmtip]{xy}

\usepackage{a4wide}
\usepackage{pxfonts}
\usepackage{wasysym}
\usepackage{amssymb,amsxtra,mathrsfs}
\usepackage{eucal}
\usepackage{setspace,upgreek}
\usepackage{mathtools}
\usepackage{url}
\usepackage[hidelinks]{hyperref}

\newcommand{\mf}[1]{\mathfrak{#1}}
\newcommand{\mc}[1]{\mathcal{#1}}

\newcommand{\ms}[1]{\mathsf{#1}}
\newcommand{\mr}[1]{\mathrm{#1}}

\newcommand{\ps}[1]{\pmb{\mathsf{#1}}}

\newcommand{\ov}{\overline}

\newcommand{\lr}[2]{\langle\, #1,\,#2\,\rangle}
\newcommand{\up}{\upharpoonright}
\newcommand{\ep}{\upvarepsilon} 
\newcommand{\ze}{\mathsf{0}} 
\newcommand{\un}{\mathsf{1}} 
\newcommand{\N}{\mathbb{N}}

\newcommand{\R}{\mathbb{R}}
\newcommand{\C}{\mathbb{C}}
\newcommand{\Z}{\mathbb{Z}}
\newtheorem{theorem}{Theorem}
\newtheorem{definition}[theorem]{Definition}

\newtheorem{proposition}[theorem]{Proposition}
\newtheorem{corollary}[theorem]{Corollary}
\newtheorem{lemma}[theorem]{Lemma}

\theoremstyle{definition}

\newtheorem{remark}
[theorem]
{Remark}
\numberwithin{theorem}{section}
\begin{document}
\title[Construction of a Natural Transformation]
{Construction of a Natural Transformation from a Classical to a Quantum $0$-Species}

\author{Benedetto Silvestri}

\email{bndsilvestri@proton.me}

\address{440 W Kelly Ave Apt 6 Jackson, WY 83001-8544 USA (via express couriers signature required)}

\urladdr{\url{https://arxiv.org/a/silvestri_b_1}}

\urladdr{https://orcid.org/0000-0002-8042-4932}

\begin{abstract}
A natural transformation $\mathfrak{J}$ between functors valued in the category $\mathfrak{Chdv}_{0}$ 
is assembled. $\mathfrak{Chdv}_{0}$ is obtained by replacing both the categories $\mathrm{ptls}$ 
and $\mathrm{ptsa}$ with the category of topological linear spaces in the defining properties 
of the category $\mathfrak{Chdv}$ introduced in one of our previous papers.
By letting a $\mathfrak{dp}$-valued functor be (classical) quantum whenever every its value is a dynamical 
pattern whose set map takes values in the set of (commutative) noncommutative topological unital $\ast-$algebras,
and letting a (classical) quantum $0$-species be a $\mathfrak{Chdv}_{0}$-valued functor factorizing through
the canonical functor from $\mathfrak{dp}$ to $\mathfrak{Chdv}_{0}$ into a (classical) quantum 
$\mathfrak{dp}$-valued functor, we have that the domain and codomain of $\mathfrak{J}$ 
are a classical and a quantum $0$-species respectively.
\end{abstract}
\maketitle
\tableofcontents
\flushbottom
\section*{Notation}
In this section we fix the notation and collect general facts we shall use often without any further mention
in the paper including the Introduction.
If $X,Y$ are topological spaces, then $\mc{C}(X,Y)$ is the set of continuous maps 
from $X$ to $Y$. If the set underlying $X$ is a subset of the set underlying $Y$, then $X\hookrightarrow Y$ 
means $\imath_{X}^{Y}\in\mc{C}(X,Y)$ with $\imath_{X}^{Y}$ the inclusion map of $X$ into $Y$.
$\mr{vct}$ is the category of complex vector spaces and linear maps, 
$\mr{top}$ is the category of topological spaces and continuous maps.
\par
All topological vector spaces are considered over $\C$,
$\mr{tls}$ is the category of topological vector spaces and continuous linear maps.
If $E,F\in\mr{tls}$, then we let  
$\mf{L}(E,F)\coloneqq\mr{Mor}_{\mr{tls}}(E,F)=\mc{C}(E,F)\cap\mr{Mor_{vct}}(E,F)$,
$\mf{L}(E)\coloneqq\mf{L}(E,E)$ and $E^{\prime}\coloneqq\mf{L}(E,\C)$. 
If $E$ and $F$ are locally convex spaces and $\mf{G}$ is a family of bounded subsets of $E$, then 
$\mf{L}_{\mf{G}}(E,F)$ is the locally convex space $\mf{L}(E,F)$ endowed with the topology 
of uniform convergence over the sets in $\mf{G}$. $\mf{L}_{s}(E,F)$, $\mf{L}_{b}(E,F)$ and $\mf{L}_{pc}(E,F)$   
stand for $\mf{L}_{\mf{G}}(E,F)$ with $\mf{G}$ respectively the family of finite, bounded and precompact 
subsets of $E$. If $p$ is a continuous seminorm on $F$ and $B$ is a bounded set of $E$, then 
$p_{B}:T\mapsto\sup_{x\in B}p(Tx)$ is a continuous seminorm of $\mf{L}_{b}(E,F)$.
$\mr{ptls}$ is the subcategory of $\mr{tls}$ of preordered topological vector spaces and linear continuous 
positive maps, $\mf{r}\in\mr{Fct}(\mr{ptls},\mr{tls})$ is the forgetful functor.
\par
For any unital $\ast$-algebra $\ms{A}$, 
$\ms{F}_{\max}$ (relative to a fixed wedge in the hermitian elements of $\ms{A}$)
is defined in \cite[p.24]{schm}, while 
$\ms{A}$-invariant non-empty subsets of $\ms{F}_{\max}$ are defined in \cite[p.25]{schm}.
If $\mc{H}$ and $\mc{G}$ are Hilbert spaces, then for every densely defined linear operator in $\mc{H}$ 
with values in $\mc{G}$, let $S^{\intercal}$ denote the adjoint of $S$. The concept of $O^{\ast}$-algebra $\mc{A}$
on a dense linear subspace $D$ of a Hilbert space is provided in \cite[Def. 2.1.6]{schm}, 
the locally convex space $D_{\mc{A}}$ is defined in \cite[Def. 2.1.1]{schm} and its topology is called the graph
topology of $\mc{A}$ on $D$. The linear space $\mc{L}(D_{\mc{A}},D_{\mc{A}}^{+})$ is defined in 
\cite[Def. 3.2.1]{schm} and the bounded topology $\tau_{b}$ on it in \cite[p.76]{schm}. 
\par
$\mr{tsa}_{0}$ is the category of topological $\ast$-algebras and continuous $\ast$-morphisms,
$\mr{tsa}$ is the category of topological unital $\ast$-algebras and continuous unit preserving $\ast$-morphisms. 
Let $\tilde{\mf{q}}_{0}\in\mr{Fct}(\mr{tsa}_{0},\mr{tls})$, $\mf{q}_{0}\in\mr{Fct}(\mr{tsa},\mr{tls})$ 
and $\mf{q}_{1}\in\mr{Fct}(\mr{tsa},\mr{top})$ be the forgetful functors.
If $A$ is a subcategory of $B$, then we let $\mr{I}_{A}^{B}$ denote the inclusion functor.
Given an object $A$ with a structure we often use, as we did above, the common abuse of language of denoting 
by $A$ each of its underlying structure. So for instance if $A$ and $B$ are topological unital $\ast$-algebras, 
$\mf{L}(A,B)$ stands for $\mf{L}(\mf{q}_{0}(A),\mf{q}_{0}(B))$ 
while $\mc{C}(A,B)$ stands for $\mc{C}(\mf{q}_{1}(A),\mf{q}_{1}(B))$. 
A topological unital sub $\ast$-algebra $A$ of $B\in\mr{tsa}$ here always means that 
$A\in\mr{tsa}$ so that $A$ is a topological subspace of $B$, $A$ is a sub $\ast$-algebra 
of $B$ and $A$ holds the same unit of $B$. Given $A\in\mr{tsa}_{0}$ we let $A_{1}\in\mr{tsa}$ be the unitization 
of $A$ \cite[p.38]{fra} whose topology by definition is the product topology.
In case $A$ is a locally convex $\ast$-algebra whose topology is generated by the set 
$S$ of seminorms, then $\tilde{S}=\{\tilde{r}\,\vert\,r\in S\}$ generates the locally convex topology of 
$A_{1}$, where $\tilde{p}(a,\lambda)\coloneqq p(a)+|\lambda|$ for every seminorm $p$ on $A$.
In particular if $p$ is a continuous seminorm on $A$, then $\tilde{p}$ is continuous on $A_{1}$.
Thus if $A\in\mr{tsa}_{0}$, and $B\in\mr{tsa}$ are both locally convex and 
$T\in\mr{Mor}_{\mr{tls}}(\tilde{\mf{q}}_{0}(A),\mf{q}_{0}(B))$, then 
\begin{equation}
\label{12040058}
\bigl((a,\lambda)\mapsto Ta+\lambda\un_{B}\bigr)\in\mr{Mor}_{\mr{tls}}(\mf{q}_{0}(A_{1}),\mf{q}_{0}(B)); 
\end{equation}
since for every continuous seminorm $q$ of $B$ there exists a continuous seminorm $p$ of $A$ such that for all 
$(a,\lambda)\in A_{1}$ we have $q(Ta+\lambda\un_{B})\leq q(Ta)+k|\lambda|\leq p(a)+k|\lambda|$ 
with $k=q(\un_{B})$, thus if $k\neq 0$, then $q(Ta+\lambda\un_{B})\leq k\widetilde{p^{\prime}}(a,\lambda)$ 
with $p^{\prime}=k^{-1}p$, otherwise $q(Ta+\lambda\un_{B})\leq\tilde{p}(a,\lambda)$. 
\par
Given two $\mr{top}$-quasi enriched categories $A$ and $B$,
let $\mr{Fct}_{\mr{top}}(A,B)$ denote the set of functors of $\mr{top}$-quasi enriched categories
from $A$ to $B$ \cite[Section 1.2]{sil2cat}.
\par
If $X$ is a locally compact space, then $\mc{K}(X)$ denotes the locally convex space of complex valued 
continuous maps on $X$ with compact support endowed with the usual inductive limit topology 
\cite[Ch. 3, \S 1, n. 1]{IntBourb}. In this paper a measure on $X$ always means an element of 
$\mc{K}(X)^{\prime}$ \cite[Ch. 3, \S 1, n. 3, Def. 2]{IntBourb}. 
\par
All manifolds are smooth and finite dimensional, hence locally compact. All vector fields are smooth.
Let $M$ be a manifold. 
$\mc{C}^{\infty}(M)$ denotes the locally convex space of complex valued smooth maps on $M$ endowed with the 
Frechet topology \cite[p.12]{wal} here denoted by $\tau^{\infty}(M)$ or simply $\tau^{\infty}$.
$\mc{D}(M)$ denotes the locally convex space of complex valued smooth maps on $M$ 
with compact support endowed with the usual inductive limit topology \cite[p.13]{wal} here denoted by
$\tau_{c}^{\infty}(M)$ or simply $\tau_{c}^{\infty}$.
We have $\mc{D}(M)\hookrightarrow\mc{C}^{\infty}(M)$ \cite[Rmk. 1.1.13]{wal}.
$\mc{D}(M)$ is a Montel space \cite[example 6, p.241]{hor} so barrelled, 
sequentially complete \cite[Thm. 1.1.11(i)]{wal} topological $\ast-$algebra \cite[28.7, 28.12]{fra} such that
$\mc{D}(M)\hookrightarrow\mc{K}(M)$ \cite[p.241]{hor}. 
Let $\lr{\mc{D}(M)}{\cdot}$ denote the natural left $\mc{C}^{\infty}(M)$-module. 
\par
If $N$ and $M$ are manifolds and $\phi:N\to M$ is a smooth proper map \cite[Def. 1.1.16]{wal}, 
then we shall consider $\phi^{\ast}$ defined on $\mc{D}(M)$, thus $\phi^{\ast}\in\mf{L}(\mc{D}(M),\mc{D}(N))$ 
\cite[Prp. 1.1.17]{wal}.
For any $k\in\Z_{\ast}^{+}$ let $\mr{DiffOp}^{k}(M,N)$ be the set of the restrictions at $\mc{D}(M)$ of the 
elements in $\mr{DiffOp}_{0}^{k}(M\times\C,N\times\C)$ 
where for every vector bundle $A$ on $M$ and $B$ on $N$, we let $\mr{DiffOp}_{0}^{k}(A,B)$ be 
the set of differential operators of order $k$ from $A$ to $B$ \cite[Def. 1.2.1]{wal}.
Set $\mr{DiffOp}^{k}(M)=\mr{DiffOp}^{k}(M,M)$.
$\lr{\mr{DiffOp}^{k}(M,N)}{\cdot}$ is naturally a left $\mc{C}^{\infty}(N)$-module, 
since $\mr{DiffOp}_{0}^{k}(A,B)$ it is so, where for every $F\in\mc{C}^{\infty}(N)$ and 
$T\in\mr{DiffOp}^{k}(M,N)$ we set $(F\cdot T):\mc{D}(M)\to\mc{D}(N)$, $h\mapsto F\cdot T(h)$.
If $T\in\mr{DiffOp}^{k}(M,N)$, then $\mr{supp}(Df)\subseteq\mr{supp}(f)$ for every $f\in\mc{D}(M)$
\cite[Rmk.1.2.2(iv)]{wal}.
We have $\mr{DiffOp}^{k}(M,N)\subset\mf{L}(\mc{D}(M),\mc{D}(N))$ \cite[Thm. 1.2.10]{wal}.
\par
Let $\mf{X}(M)$ be the set of vector fields of $M$, if $U\in\mf{X}(M)$, then let $\mathsterling_{U}$ 
be the restriction at $\mc{D}(M)$ of the Lie derivative on $\mc{C}^{\infty}(M)$ associated with $U$
here denoted by $\mathsterling_{U}^{\divideontimes}$; so $\mathsterling_{U}f=Uf$ for every $f\in\mc{D}(M)$ and 
in particular $\mathsterling_{U}\in\mr{DiffOp}^{1}(M)$.
If $V\in\mf{X}(N)$, $U\in\mf{X}(M)$, $\phi:N\to M$ is smooth and $V$ and $U$ are $\phi$-related, then 
$\mathsterling_{V}\circ\phi^{\ast}=\phi^{\ast}\circ\mathsterling_{U}$.
Whenever $U$ is complete, we let $\uptheta^{U}:\R\to\mr{Diff}(M)$ be the flow on $M$ generated by $U$ 
and for every $t\in\R$ let $\upeta_{M}^{U}(t)\coloneqq(\uptheta^{U}(-t))^{\ast}\in\mf{L}(\mc{D}(M))$ 
namely the map $f\mapsto f\circ\uptheta_{-t}^{U}$.
\par
Let $\mc{M}=(M,g)$ be a semi-Riemannian manifold, thus for every $f\in\mc{C}^{\infty}(M)$ let 
$\mr{grad}_{\mc{M}}(f)$ be the gradient of $f$ w.r.t. $g$, thus $\mr{grad}_{\mc{M}}(f)\in\mf{X}(M)$ such that 
$\lr{\mr{grad}_{\mc{M}}(f)}{Y}_{\mc{M}}=\mathsterling_{Y}(f)$ for every $Y\in\mf{X}(M)$ where 
$\lr{\cdot}{\cdot}_{\mc{M}}:\mf{X}(M)\times\mf{X}(M)\to\mc{C}^{\infty}(M)$ is the 
$\mc{C}^{\infty}(M)$-bilinear map corresponding to the metric $g$.
Let $\mu_{g}$ denote the measure on $M$ associated via 
\cite[Thm. 4.7]{lang} with the density relative to $g$ \cite[Prp. 2.1.15(ii)]{wal}, 
set $\mc{H}_{g}\coloneqq\mr{L}^{2}(M,d\mu_{g})$. We have $\mc{K}(M)\hookrightarrow\mc{H}_{g}$ since for 
instance \cite[Thm. 1.1.11(iv)]{wal}, since $\mu_{g}$ is continuous on $\mc{K}(M)$ and since 
$\|f^{\ast}f\|=\|f\|^{2}$ for every compact $K$ and every $f\in\mc{C}(X,K)$, being a $\mr{C}^{\ast}-$algebra
the normed space $\mc{C}(X,K)$ of complex valued continuous maps on $X$ with support in $K$ endowed with the 
$\sup-$norm. The inclusion $\mc{K}(M)\hookrightarrow\mc{H}_{g}$ is dense 
\cite[Ch. 4, \S 3, n. 4, Def. 2]{IntBourb} as well the inclusion $\mc{D}(M)\hookrightarrow\mc{H}_{g}$.
If $(N,g^{\prime})$ is a semi-Riemannian manifold and $D\in\mr{DiffOp}^{k}(M,N)$, 
then $D^{\intercal}$ is well-set since $\mc{D}(M)$ is dense in $\mc{H}_{g}$, 
moreover by \cite[Thm. 1.2.15]{wal} we deduce that $\mc{D}(N)\subseteq\mr{Dom}(D^{\intercal})$ and 
$D^{\intercal}\mc{D}(N)\subseteq\mc{D}(M)$.
If $\phi:N\to M$ is a smooth diffeomorphism such that $\phi^{\ast}g=g^{\prime}$, thus $\phi^{\ast}$ 
(on $\mc{D}(M)$) extends to a unitary operator from $\mc{H}_{g}$ onto $\mc{H}_{g^{\prime}}$ still denoted 
$\phi^{\ast}$ such that $(\phi^{\ast})^{\intercal}=(\phi^{-1})^{\ast}$.
\par
Here by a $C_{0}$-semigroup on a topological vector space $Y$ is meant a map $U\in\mc{C}(\R_{+},\mf{L}_{s}(Y))$
such that $U(s+t)=U(s)U(t)$ for all $t,s\in\R_{+}$ and $U(0)=\un$. In addition by letting $\tau$ the topology
of $Y$, $U$ is called $\tau-$equicontinuous or simply equicontinuous if $\{U(t)\,\vert\,t\in\R_{+}\}$
is a $(\tau,\tau)$-equicontinuous set. Similar definitions for a $C_{0}$-group by replacing $\R_{+}$ with $\R$.
\par
If $X$ is a sequentially complete locally convex space with topology $\tau$ and $T\in\mf{L}(X)$ such that 
$\{T^{n}\,\vert\,n\in\Z_{+}^{\ast}\}$ is $(\tau,\tau)$-equicontinuous, then it is well-known that
there exists a $C_{0}$-semigroup $\exp_{X}^{T}$ on $X$ such that: 
\begin{enumerate}
\item
$T$ is the infinitesimal $\tau-$generator of $\exp_{X}^{T}$; 
\item
$\exp_{X}^{T}(t)x=\sum_{k=0}^{\infty}\frac{(tT)^{k}}{k!}x$ convergence in $X$ for every 
$t\in\R_{+}$ and $x\in X$;
\item
by letting $\ov{\exp}_{X}^{T}:\R_{+}\ni t\mapsto\exp(-t)\exp_{X}^{T}(t)$ we have that 
$\ov{\exp}_{X}^{T}$ is an equicontinuous $C_{0}$-semigroup on $X$.
\end{enumerate}
Since the equicontinuity hypothesis it is clear that the series in (2) extends to $t\in\R$, so $\exp_{X}^{T}$
extends to a $C_{0}$-group on $X$ still denoted by the same symbol, moreover 
$\underline{\exp}_{X}^{T}$ is an equicontinuous $C_{0}$-semigroup on $X$, where
$\underline{\exp}_{X}^{T}:\R_{+}\ni t\mapsto\exp(-t)\exp_{X}^{T}(-t)$.
\section*{Introduction}
In \cite[Cor.1.6.43]{sil2cat} and the discussion after we have shown that the existence of a 
natural transformation, from the classical gravity species $\mr{a}^{4}$ to a strict quantum gravity species, 
satisfying certain constraints would render the 
dark energy hypothesis unnecessary in explaining the actual cosmic acceleration. 
This paper is one step toward a better understanding of the way to construct such a natural transformation.
\par
In order to describe our results we need some additional terminology.
\par
First of all we recall that an object of the category $\mf{dp}$ of dynamical patterns
(\cite[Cor. 1.4.5 and Def. 1.4.1]{sil2cat}) is a functor of $\mr{top}$-quasi enriched categories 
(i.e. a functor whose morphism map is \emph{continuous})
valued in the $\mr{top}$-quasi enriched category $\mr{tsa}$ of unital topological $\ast$-algebras
(enriched by endowing the morphism set of every two objects of $\mr{tsa}$ 
with the topology of simple convergence). A morphism of dynamical patterns is a couple $(f,\mr{T})$ 
formed by a functor $f$ of $\mr{top}$-quasi enriched categories from the domain of the second dynamical pattern 
to the domain of the first one, and by a natural transformation $\mr{T}$ from the composition of the 
first dynamical pattern with $f$ to the second dynamical pattern. 
\par
Next let $\mf{Chdv}_{0}$ be the category introduced in Prp. \ref{12131143}
and obtained by replacing in the defining properties 
of the category $\mf{Chdv}$ (\cite[Cor. 1.4.18 and Def. 1.4.17]{sil2cat}) both the categories $\mr{ptls}$ and 
$\mr{ptsa}$ with the category of topological linear spaces $\mr{tls}$. Similarly at $\ps{\Uppsi}$ there exists 
the (canonical) functor $\ps{\Uppsi}_{0}$ from $\mf{dp}$ to $\mf{Chdv}_{0}$. 
A $0$-species is a functor valued in $\mf{Chdv}_{0}$ which factorizes through $\mf{dp}$ 
(Def. \ref{12041523}) in particular a $0$-species is a $1$-cell of the $2$-category 
$2-\mf{dp}$. 
A dynamical pattern is called quantum (respectively classical) if all its values are 
noncommutative (respectively commutative) algebras.
A functor valued in $\mf{dp}$ is called quantum (respectively classical) if all its values are quantum 
(respectively classical) dynamical patterns. Finally a $0$-species is called quantum (respectively classical) 
if it factorizes through $\ps{\Uppsi}_{0}$ into a quantum (respectively classical) functor valued in $\mf{dp}$.
Thus we have what follows.
\par
In \textbf{Thm. \ref{11280957}} and \textbf{Thm. \ref{11290640}} we construct two functors valued in $\mf{dp}$, 
the first $\mf{x}$ classical and the second $\mf{z}$ quantum.
\par
Then in \textbf{Thm. \ref{12031824}} we establish our main result: The existence of the natural transformation 
$\mf{J}$ from the classical $0$-species $\ms{x}$ to the quantum $0$-species $\ms{z}$,
where $\ms{x}$ factorizes to the left and to the right through $\mf{x}$ and $\ms{z}$ factorizes
to the left and to the right through $\mf{z}$.
\par
Now the following observation is worthwhile.
\par 
Since in the present paper we are decisively dealing with the categories $\mf{dp}$ and $\mf{Chdv}_{0}$, 
specifically with the construction of the functors $\mf{x}$ and $\mf{z}$
and the construction of the natural transformation $\mf{J}$,
statements concerning \textbf{continuity} acquire a distinctive value. 
Specifically we refer to:
\begin{enumerate}
\item
The $C_{0}$-continuity of the semigroup $\Gamma_{\mc{M}}^{U}$ (Thm. \ref{11191752}\eqref{11191752st2}) 
at the core of the object map of $\mf{z}$.
\item
The continuity of the $\ast$-morphism $\mc{T}(\phi)$ (Thm. \ref{14111827}\eqref{14111827st2})
at the core of the morphism map of $\mf{z}$.
\item
The continuity of the map $f\mapsto\mathsterling_{\mr{grad}_{\mc{M}}(f)}$ (Cor. \ref{12040944})
at the core of $\mf{J}$.
\end{enumerate}
\par
In the remaining of this introduction we shall briefly outline the main steps to arrive at our main result.
\par
\textbf{Thm. \ref{14111827}} and \textbf{Thm. \ref{11191752}} are the main results of section \ref{12111211}. 
In Thm. \ref{14111827}\eqref{14111827st1} we prove that $\mf{B}(\mc{M})$ is a unital topological $\ast$-algebra 
and in Thm. \ref{14111827}\eqref{14111827st2} we prove that $\phi$ implements via $\mc{T}$ a morphism of 
unital topological $\ast$-algebras.
In Thm. \ref{11191752}\eqref{11191752st2} we establish the existence of $\Gamma_{\mc{M}}^{U}$ 
a $C_{0}$-group on $\mf{B}(\mc{M})$ of $\ast-$automorphisms and in Thm. \ref{11191752}\eqref{11191752st3}
we prove that $\Gamma$ and $\mc{T}$ are equivariant namely \eqref{12111055} holds true.
Here $\mc{M}=(M,g)$ and $\mc{N}=(N,g^{\prime})$ are semi-Riemannian manifolds and 
$\phi:N\to M$ is a smooth diffeomorphism such that $\phi^{\ast}g=g^{\prime}$.
Our construction of $\mf{B}$ is calibrated to ensure that $\mc{T}$ and $\Gamma$ possess the above properties.
\par
We start by defining $\exp_{M}^{U}$ as the exponential $C_{0}$-group, on the sequentially complete locally
convex space $\mc{D}(M)$ (remember $\mc{D}(M)$ is endowed with the inductive limit topology 
$\tau_{c}^{\infty}(M)$), generated by $\mathsterling_{U}$ provided $\{\mathsterling_{U}^{k}\}_{k\in\Z_{+}}$
be $(\tau_{c}^{\infty},\tau_{c}^{\infty})$-equicontinuous, and let $\Lambda_{M}^{U}$ be the corresponding 
action on $\mf{L}(\mc{D}(M))$ namely
\begin{equation*}
\Lambda_{M}^{U}:t\mapsto(T\mapsto\exp_{M}^{U}(t)\circ T\circ\exp_{M}^{U}(-t)).
\end{equation*}
Next we define the set $\ms{B}(\mc{M})$ underlying $\mf{B}(\mc{M})$ in Def. \ref{10311558} as the 
subset of those linear and continuous operators on $\mc{D}(M)$ whose Hilbert space adjoint 
in $\mc{H}_{g}$ is such that its domain contains $\mc{D}(M)$, maps $\mc{D}(M)$ into itself and 
its restriction to $\mc{D}(M)$ is continuous:
\begin{multline*}
\ms{B}(\mc{M})\coloneqq
\{
T\in\mf{L}(\mc{D}(M))\,\vert\,\mc{D}(M)\subseteq\mr{Dom}(T^{\intercal}),\\ 
T^{\intercal}\mc{D}(M)\subseteq\mc{D}(M), T^{\dagger}\coloneqq 
T^{\intercal}\up\mc{D}(M)\in\mf{L}(\mc{D}(M))
\};
\end{multline*}
where $T^{\intercal}$ is the $\mc{H}_{g}-$adjoint of the operator $T$.
In Prp. \ref{11241301} we show that $\ms{B}(\mc{M})$ 
is a $O^{\ast}$-algebra on $\mc{D}(M)$, $\mc{D}(M)$ seen in this context as a dense linear subspace of 
$\mc{H}_{g}$, in particular $\ms{B}(\mc{M})$ is a unital $\ast$-algebra. 
Then in Def. \ref{10311600} we define a set of functionals $F_{\mc{M}}$ over $\ms{B}(\mc{M})$ and prove in 
Prp. \ref{10311601} that $F_{\mc{M}}$ is a $\ms{B}(\mc{M})$-invariant non-empty subset of $\ms{F}_{\max}$
(relative to the wedge of finite sums of positive elements of $\ms{B}(\mc{M})$). 
This result along with the general result \cite[Lemma 1.5.7]{schm} 
applied to our $\ast$-algebra $\ms{B}(\mc{M})$ and our set $F_{\mc{M}}$,
enables us to show in Thm. \ref{14111827}\eqref{14111827st1} that 
\begin{equation}
\label{12121507}
\mf{B}(\mc{M})\coloneqq\ms{B}(\mc{M})[\tau_{\mc{M}}]\in\mr{tsa};
\end{equation}
where $\tau_{\mc{M}}$ (Def. \ref{14111822}) is the locally convex topology on $\ms{B}(\mc{M})$ generated 
by the following set of seminorms 
\begin{equation*}
\begin{cases}
\{q^{B}\,\vert\,B\in\mr{Bounded}(\mc{D}(M))\};
\\
q^{B}:\ms{B}(\mc{M})\ni T\mapsto\sup_{f\in B}|\lr{f}{Tf}_{\mc{H}_{g}}|.
\end{cases}
\end{equation*}
There is another $\mr{tsa}$-structure over $\ms{B}(\mc{M})$, indeed
$\ms{B}(\mc{M})$ endowed with the topology relative to the bounded topology on
$\mc{L}(\mc{D}(M)_{\ms{B}(\mc{M})},\mc{D}(M)_{\ms{B}(\mc{M})}^{+})$
is a unital topological $\ast$-algebra and this topology is stronger than $\tau_{\mc{M}}$ (Prp. \ref{11271148}). 
Next by letting 
\begin{equation}
\label{12131428}
\mc{T}(\phi)(T)\coloneqq\phi^{\ast}\circ T\circ(\phi^{-1})^{\ast};
\end{equation}
we prove in Thm. \ref{14111827}\eqref{14111827st2} that 
\begin{equation}
\label{12121510}
\mc{T}(\phi)\in\mr{Mor_{tsa}}(\mf{B}(\mc{M}),\mf{B}(\mc{N})).
\end{equation}
In Cor. \ref{11231634} we prove that 
$\Lambda_{M}^{U}$ is a $C_{0}$-group on the locally convex space $\mf{L}_{b}(\mc{D}(M))$ 
of continuous linear maps on $\mc{D}(M)$ endowed with the topology of uniform
convergence over the bounded subsets of $\mc{D}(M)$ namely 
\begin{equation}
\label{12120927}
\begin{cases}
\Lambda_{M}^{U}\in\mc{C}\bigl(\R,\mf{L}_{s}(\mf{L}_{b}(\mc{D}(M)))\bigr);
\\
\Lambda_{M}^{U}\text{ is a one-parameter group}.
\end{cases}
\end{equation}
This result is a consequence of Lemma \ref{11231514}, a more general result important in its own, 
enlightening the twofold essential role palyed by the Montel space $\mc{D}(M)$
in obtaining Cor. \ref{11231634}: the first directly by its definition, the second permitting to use of the 
Banach-Steinhaus Thm. since any Montel space is barrelled.
\par
In Def. \ref{11241119} we define the category $\mf{vf}_{0}$ of the couples $(\mc{M},U)$ with the following
properties: $\mc{M}=(M,g)$ is a semi-Riemannian manifold, $U$ is a vector field of $M$ such that 
$\{\mathsterling_{U}^{n}\,\vert\,n\in\Z_{+}^{\ast}\}$ is $(\tau_{c}^{\infty},\tau_{c}^{\infty})$-equicontinuous,
and the following property of invariance holds true
\begin{equation}
\label{12121530}
\mu_{g}\circ\mathsterling_{U}=\ze.
\end{equation}
While $\phi$ is a morphism from $(\mc{M},U)$ to $(\mc{N},V)$ iff $\phi:N\to M$ is smooth, 
$\phi^{\ast}g=g^{\prime}$, and $U$ and $V$ are $\phi$-related with $\mc{N}=(N,g^{\prime})$.
$\mf{vf}$ is the subcategory of $\mf{vf}_{0}$ with the same object set and diffeomorphisms as morphisms.
Now the reason of introducing the above categories stands on 
Thm. \ref{11191752}(\ref{11191752st2},\ref{11191752st3}) establishing that whenever
$(\mc{M},U),(\mc{N},V)\in\mf{vf}$ and $\phi\in\mr{Mor}_{\mf{vf}}((\mc{M},U),(\mc{N},V))$, 
$\Lambda_{M}^{U}$ restricts to a $C_{0}$-group $\Gamma_{\mc{M}}^{U}$ on $\mf{B}(\mc{M})$ of $\ast-$automorphisms 
such that $\mc{T}$ and $\Gamma$ are equivariant, namely  
\begin{equation}
\label{12120918}
\begin{cases}
\Gamma_{\mc{M}}^{U}\in\mc{C}\bigl(\R,\mf{L}_{s}(\mf{B}(\mc{M}))\bigr),
\\
\Gamma_{\mc{M}}^{U}\text{ is a one-parameter group of $\ast-$automorphisms};
\end{cases}
\end{equation}
and for every $t\in\R$
\begin{equation}
\label{12111055}
\mc{T}(\phi)\circ\Gamma_{\mc{M}}^{U}(t)=\Gamma_{\mc{N}}^{V}(t)\circ\mc{T}(\phi).
\end{equation}
Let us outline the essential steps yielding to \eqref{12120918}. 
Firstly in Cor. \ref{11261248} we show that whenever $(\mc{M},U)\in\mf{vf}_{0}$, the group $\exp_{M}^{U}$ 
extends to a unique $C_{0}$-group $\mf{exp}_{\mc{M}}^{U}$ on $\mc{H}_{g}$ of unitary operators whose 
infinitesimal generator extends $\mathsterling_{U}$. 
\par
It is worthwhile remarking that the unitary extension is essentially due to \eqref{12121530} 
(proof of Lemma \ref{11240851}). Thus Cor. \ref{11261248} ensures that $\Lambda_{M}^{U}$ restricts to 
a group $\Gamma_{\mc{M}}^{U}$ on $\mf{B}(\mc{M})$ of $\ast-$automorphisms (Cor. \ref{11241005}).
\par
Now in the fundamental Lemma \ref{11250550} we prove that the topology 
$\tau_{\mc{M}}$ is generated by a collection of seminorms extending to
$\mf{L}_{b}(\mc{D}(M))$-continuous seminorms. While $\Gamma_{\mc{M}}^{U}(t)$ is a continuous linear map on 
$\mf{B}(\mc{M})$ since $\exp_{M}^{U}(t)$ maps bounded sets into bounded sets.
Therefore \eqref{12120927} implies that $\Gamma_{\mc{M}}^{U}$ is a $C_{0}$-group on $\mf{B}(\mc{M})$. 
We remark that in showing Lemma \ref{11250550} the fact that $\mc{D}(M)$ is barrelled is essential. 
\par
In addition to the above results, by an application of the Banach-Steinhaus Thm. and of the fact that 
$\mc{D}(M)$ is specifically a Montel space, in Thm. \ref{11191752}\eqref{11191752st1} we prove that 
$\exp_{M}^{U}:\R\to\mf{U}(\mc{M})$ is a continuous morphism of groups, where $\mf{U}(\mc{M})$ is the 
group of unitary elements of $\mf{B}(\mc{M})$ endowed with the relative topology.
\par
In conclusion of section \ref{12111211} we determine in Prp. \ref{12131143}
the category $\mf{Chdv}_{0}$ and 
the functor $\ps{\Uppsi}_{0}$, while $0$-species are introduced in Def. \ref{12041523}. 
\par
\textbf{Thm. \ref{11280957}} and \textbf{Thm. \ref{11290640}} are the main results of section \ref{12111459}, 
where by using Thm. \ref{14111827} and Thm. \ref{11191752}, we construct two functors $\mf{x}$ and $\mf{z}$
from the category $\mf{vf}$ to the category of dynamical patterns $\mf{dp}$, classical $\mf{x}$
and quantum $\mf{z}$. 
\par
Let us delineate what above said for the more interesting quantum functor $\mf{z}$, but first of all 
we outline the main structures involved.
\par
For every $(\mc{M},U)\in\mf{vf}$ let $\lr{\mc{M}}{U}$ be the $\mr{top}$-quasi enriched category of subsets 
of $M$ such that for all $X,Y\in\lr{\mc{M}}{U}$ we have
$\mr{Mor}_{\lr{\mc{M}}{U}}(X,Y)=\{(X,Y)\}\times\mr{mor}_{\lr{\mc{M}}{U}}(X,Y)$ 
endowed with the topology inherited by $\R$ where 
\begin{equation*}
\mr{mor}_{\lr{\mc{M}}{U}}(X,Y)=\{t\in\R\,\vert\,\exp_{M}^{U}(t)\mc{D}(M,X)=\mc{D}(M,Y)\};
\end{equation*}
and $\mc{D}(M,X)$ is the topological sub $\ast$-algebra of $\mc{D}(M)$ of those maps whose support is 
contained in $X$. 
Next let $\mf{B}(\mc{M},X)$ be the topological unital sub $\ast$-algebra 
of $\mf{B}(\mc{M})$ of those $T$ such that $T\mc{D}(M,X)\subseteq\mc{D}(M,X)$
and $T^{\dagger}\mc{D}(M,X)\subseteq\mc{D}(M,X)$.
Thus we can define the maps $(\mc{F}_{\lr{\mc{M}}{U}})_{o}$ 
and $(\mc{F}_{\lr{\mc{M}}{U}})_{m}$ on the object and morphism set of $\lr{\mc{M}}{U}$ respectively as 
\begin{equation*}
\begin{cases}
(\mc{F}_{\lr{\mc{M}}{U}})_{o}:X\mapsto\mf{B}(\mc{M},X),
\\
(\mc{F}_{\lr{\mc{M}}{U}})_{m}:((X,Y),t)\mapsto
\bigl(\mf{B}(\mc{M},X)\to\mf{B}(\mc{M},Y)\quad T\mapsto\Gamma_{\mc{M}}^{U}(t)T\bigr).
\end{cases}
\end{equation*}
While for every $(\mc{M},U),(\mc{N},V)\in\mf{vf}$ and $\phi\in\mr{Mor}_{\mf{vf}}((\mc{M},U),(\mc{N},V))$
we can set the maps $f_{o}^{\phi}$ and $f_{m}^{\phi}$ over the object and the morphism set of $\lr{\mc{N}}{V}$ 
respectively such that 
\begin{equation*}
\begin{cases}
f_{o}^{\phi}:Y\mapsto\phi(Y);
\\
f_{m}^{\phi}:((Y,Z),s)\mapsto((\phi(Y),\phi(Z)),s).
\end{cases}
\end{equation*}
and define the map $\mf{T}$ over the morphism set of $\mf{vf}$ such that 
\begin{equation*}
\mf{T}^{\phi}:Y\mapsto
\bigl(\mf{B}(\mc{M},\phi(Y))\to\mf{B}(\mc{N},Y)\quad T\mapsto\mc{T}(\phi)T\bigr).
\end{equation*}
Thus we are able to define $\mf{z}$ on the category $\mf{vf}$ such that 
\begin{equation*}
\begin{cases}
\mf{z}_{o}:(\mc{M},U)\mapsto\lr{\lr{\mc{M}}{U}}{\mc{F}_{\lr{\mc{M}}{U}}};
\\
\mf{z}_{m}:\phi\mapsto(f^{\phi},\mf{T}^{\phi}).
\end{cases}
\end{equation*}
Now we have to see that effectively 
\begin{equation}
\label{12131600}
\mf{z}\in\mr{Fct}(\mf{vf},\mf{dp}).
\end{equation}
What happens is that \eqref{12120918} is the core of the proof that the object map $\mf{z}_{o}$ 
is well-set namely 
\begin{equation*}
\mc{F}_{\lr{\mc{M}}{U}}\in\mr{Fct}_{\mr{top}}(\lr{\mc{M}}{U},\mr{tsa});
\end{equation*}
that Lemma \ref{11011940} implies that the first component of the morphism map $\mf{z}_{m}$ is well-set, namely
\begin{equation*}
f^{\phi}\in\mr{Fct_{top}}(\lr{\mc{N}}{V},\lr{\mc{M}}{U});
\end{equation*}
that \eqref{12121510} and the equivariance \eqref{12111055} 
are the core of the proof that also the second component of $\mf{z}_{m}$ is well-set, namely 
\begin{equation}
\label{12111511}
\mf{T}^{\phi}\in\mr{Mor}_{\mr{Fct}(\lr{\mc{N}}{V},\mr{tsa})}
(\mc{F}_{\lr{\mc{M}}{U}}\circ f^{\phi},\mc{F}_{\lr{\mc{N}}{V}});
\end{equation}
finally that $\mc{T}(\phi\circ\psi)=\mc{T}(\psi)\circ\mc{T}(\phi)$ 
implies that $\mf{z}_{m}$ preserves the morphism composition, and \eqref{12131600} follows.
\par
About the classical functor $\mf{x}$ the main novelties and advantages with respect to the functor $\mf{a}$ 
constructed in \cite[Thm. 1.6.24]{sil2cat} are represented by two facts: 
Firstly in order to construct a group associated with a vector field $U$, 
here $U$ needs not to be complete, rather we require $\{\mathsterling_{U}^{k}\,\vert\,k\in\Z_{+}\}$ 
to be $(\tau_{c}^{\infty},\tau_{c}^{\infty})$-equicontinuous, 
by obtaining in this way the additional $C_{0}$-property of $\exp_{M}^{U}$.
Secondly here we select a specific topology on $\mc{D}(M,X)$, then by force on its unitization $\mc{D}_{1}(M,X)$,
by de facto avoiding the problem of introducing what in \cite[Def. 1.6.18]{sil2cat} we called a $\mr{vf}$-topology.
This because for what just above said and since $\phi^{\ast}$ is $(\tau_{c}^{\infty},\tau_{c}^{\infty})$-
continuous, the $\tau_{c}^{\infty}$-topology satisfies the requirements of a 
$\mr{vf}$-topology provided $\upeta_{M}^{U}$, 
the adjoint on $\mc{D}(M)$ of the flow on $M$ generated by a complete vector field $U$ of $M$, 
be replaced by the group $\exp_{M}^{U}$.
Said that the construction of $\mf{x}$ mimics the one of $\mf{a}$, 
by replacing $\upeta_{M}^{U}$ with $\exp_{M}^{U}$.
At the end of this Introduction we shall see that in special cases $\exp_{M}^{U}$ equals $\upeta_{M}^{U}$.
\par
\textbf{Thm. \ref{12031824}} establishes the main result of section \ref{12111615} and of the entire work, 
namely the existence of the natural transformation 
\begin{equation}
\label{12121947}
\mf{J}\in\mr{Mor}_{\mr{Fct}(\mf{Vf},\mf{Chdv}_{0})}(\ms{x},\ms{z});
\end{equation}
between the classical $0$-species $\ms{x}\coloneqq\ps{\Uppsi}_{0}\circ\mf{x}\circ\mr{I}_{\mf{Vf}}^{\mf{vf}}$ and 
the quantum $0$-species $\ms{z}\coloneqq\ps{\Uppsi}_{0}\circ\mf{z}\circ\mr{I}_{\mf{Vf}}^{\mf{vf}}$, uniquely 
determined by
\begin{equation*}
\begin{aligned}
\mf{J}:\mr{Obj}(\mf{Vf})\ni(\mc{M},U)
&\mapsto
\mf{J}(\mc{M},U)
=
(\un_{\lr{\mc{M}}{U}},\mc{J}_{(\mc{M},U)}^{\dagger},\mc{J}_{(\mc{M},U)});
\\
\mc{J}_{(\mc{M},U)}:\mr{Obj}(\lr{\mc{M}}{U})
&\ni
X\mapsto\mc{Z}_{(\mc{M},U)}^{X};
\\
\mc{J}_{(\mc{M},U)}^{\dagger}:\mr{Obj}(\lr{\mc{M}}{U})
&\ni
X
\mapsto(\mc{Z}_{(\mc{M},U)}^{X})^{\dagger};
\end{aligned}
\end{equation*}
where 
\begin{equation*}
\begin{aligned}
\mc{Z}_{(\mc{M},U)}^{X}:\mc{D}_{1}(M,X)&\to\mf{B}(\mc{M},X)
\\
(f,\lambda)&\mapsto\mathsterling_{[\mr{grad}_{\mc{M}}(f),U]}+\lambda\un;
\end{aligned}
\end{equation*}
and where $\mf{Vf}$ is the full subcategory of $\mf{vf}$ of those $(\mc{M},U)$ for which there exists a frame 
$\{E_{i}\}$ of orthonormal fields of $\mc{M}$ such that 
\begin{equation}
\label{12160257}
(\forall i\in[1,\mr{dim}M]\cap\Z)([U,E_{i}]=\ze).
\end{equation}
Three are the fundamental steps to establish \eqref{12121947}.
\par
First of all Cor. \ref{12040944} by stating that 
\begin{equation*}
\bigl(f\mapsto\mathsterling_{\mr{grad}_{\mc{M}}(f)}\bigr)\in\mf{L}(\mc{D}(M,X),\mf{B}(\mc{M},X));
\end{equation*}
ensures that $\mc{J}_{(\mc{M},U)}(X)$ is a continuous linear map. 
\par
Then what right now stated and \textbf{Thm. \ref{11301607}} by establishing that
\begin{equation}
\label{12120840}
\exp_{M}^{U}(t)\circ\mathsterling_{[\mr{grad}_{\mc{M}}(f),U]}
=
\mathsterling_{[(\mr{grad}_{\mc{M}}\circ\exp_{M}^{U}(t))(f),U]}\circ\exp_{M}^{U}(t);
\end{equation}
ensure that 
\begin{equation*}
\mc{J}_{(\mc{M},U)}\in\mr{Mor}_{\mr{Fct}(\lr{\mc{M}}{U},\mr{tls})}
(\mf{q}_{0}\circ\mr{F}_{\lr{\mc{M}}{U}},\mf{q}_{0}\circ\mc{F}_{\lr{\mc{M}}{U}});
\end{equation*}
which together its adjoint imply 
\begin{equation}
\label{12121945}
\mf{J}(\mc{M},U)\in\mr{Mor}_{\mf{Chdv}_{0}}(\ms{x}(\mc{M},U),\ms{z}(\mc{M},U)).
\end{equation}
It is in order to determine \eqref{12120840} that we require the use of the category $\mf{Vf}$ 
rather than $\mf{vf}$. Specifically hypothesis \eqref{12160257} ensures that the following term 
\begin{equation*}
\sum_{i=1}^{n}\ep_{i}\cdot\mathsterling_{E_{i}}(f)\mathsterling_{[E_{i},U]}
\end{equation*}
in the right side of \eqref{12091643} vanishes.
\par
Finally Lemma \ref{12031345} states that 
\begin{equation*}
\phi^{\ast}\circ\mathsterling_{\mr{grad}_{\mc{M}}(f)}
=
\mathsterling_{(\mr{grad}_{\mc{N}}\circ\phi^{\ast})(f)}\circ\phi^{\ast}.
\end{equation*}
which together \eqref{12121945} represent the core of the proof of the commutativity of the following diagram 
in the category $\mf{Chdv}_{0}$
\begin{equation*}
\xymatrix{
\ms{x}(\mc{N},V)
\ar[rr]^{\mf{J}(\mc{N},V)}
&&
\ms{z}(\mc{N},V)
\\
&&
\\
\ms{x}(\mc{M},U)
\ar[rr]^{\mf{J}(\mc{M},U)}
\ar[uu]^{\ms{x}(\phi)}
&&
\ms{z}(\mc{M},U)
\ar[uu]_{\ms{z}(\phi)}}
\end{equation*}
and \eqref{12121947} follows.
\par
\textbf{Cor. \ref{12061539}} is the main result of the closing section \ref{12120940}, where 
under the hypothesis that $U$ is complete and an additional equicontinuity condition on $\mathsterling_{U}$,
we answer in Cor. \ref{12061539}\eqref{12061539st2} the natural question in the affirmative on whether 
$\exp_{M}^{U}$ equals the adjoint action on $\mc{D}(M)$ of the flow on $M$ generated by $U$.
In the same section we also prove in Lemma \ref{11301957} that under the obvious additional equicontinuity 
request over $\mathsterling_{U}^{\divideontimes}$ the Lie derivative of $U$ on $\mc{C}^{\infty}(M)$, 
the exponential one-parameter group $\mr{Exp}_{M}^{U}$ 
generated by $\mathsterling_{U}^{\divideontimes}$ extends $\exp_{M}^{U}$.
As a result in Prp. \ref{12001100}\eqref{12001100st2} we obtain that 
$\Lambda_{M}^{U}(t)$ restricts to a morphism $\mc{A}\to\mc{A}_{t}$ of left $\mc{C}^{\infty}(M)$-modules
where $\mc{A}\subset\mf{L}(\mc{D}(M))$ is naturally a left $\mc{C}^{\infty}(M)$-module 
such that $\Lambda_{M}^{U}(t)\mc{A}\subseteq\mc{A}$ while $\mc{A}_{t}$ is the 
left $\mc{C}^{\infty}(M)$-module whose underlying group is $\mc{A}$ and external law is given by 
$F\cdot Q\mapsto\mr{Exp}_{M}^{U}(t)(F)\cdot Q$.
\section{Construction of the topological $\ast$-algebra $\mf{B}(\mc{M})$ and the $C_{0}$-group
$\Gamma_{\mc{M}}^{U}$}
\label{12111211}
\begin{definition}
\label{11011937}
Define $\mr{vf}^{\star}$ to be the category such that its object set is the set of the couples $(M,U)$
where $M$ is a manifold and $U$ is a vector field on $M$ such that 
$\{\mathsterling_{U}^{n}\,\vert\,n\in\Z_{+}^{\ast}\}$ is $(\tau_{c}^{\infty},\tau_{c}^{\infty})$-equicontinuous. 
For every $(M,U),(N,V)\in\mr{vf}^{\star}$, 
$\mr{Mor_{vf^{\star}}}((M,U),(N,V))$ is the set of proper smooth maps $\phi:N\to M$ so that $U$ and $V$ are 
$\phi$-related, while for every $(Q,K)\in\mr{vf}^{\star}$ and $\psi\in\mr{Mor_{vf^{\star}}}((N,V),(Q,K))$ 
we set $\psi\circ_{\mr{vf}^{\star}}\phi\coloneqq\phi\circ\psi$.
\end{definition}
Since $\mc{D}(M)$ is sequentially complete we can set the following 
\begin{definition}
\label{11231428}
Let $(M,U)\in\mr{vf}^{\star}$ define 
\begin{equation*}
\exp_{M}^{U}\coloneqq\exp_{\mc{D}(M)}^{\mathsterling_{U}},
\end{equation*}
set $\ov{\exp}_{M}^{U}:\R_{+}\ni t\mapsto\exp(-t)\exp_{M}^{U}(t)$ and 
$\underline{\exp}_{M}^{U}:\R_{+}\ni t\mapsto\exp(-t)\exp_{M}^{U}(-t)$.
Moreover define
$\Lambda_{M}^{U}:\R\to\mr{End_{vct}}(\mf{L}(\mc{D}(M)))$ such that 
\begin{equation*}
\Lambda_{M}^{U}:t\mapsto(T\mapsto\exp_{M}^{U}(t)\circ T\circ\exp_{M}^{U}(-t)).
\end{equation*}
\end{definition}
\begin{remark}
\label{11240847}
Let $(M,U)\in\mr{vf}^{\star}$. Thus for every constant map $c$ on $M$ and $f\in\mc{D}(M)$ we have 
$\exp_{M}^{U}(t)(c\cdot f)=c\cdot\exp_{M}^{U}(t)(f)$, 
since $\mathsterling_{U}(c\cdot f)=c\cdot\mathsterling_{U}(f)$
being $\mathsterling_{U}^{\divideontimes}(c)=\ze$,
by $\mc{D}(M)\hookrightarrow\mc{C}^{\infty}(M)$ and since $\mc{C}^{\infty}(M)$ is a topological algebra.
Moreover $\exp_{M}^{U}$ is a group of $\ast$-automorphisms of $\mc{D}(M)$ indeed its
infinitesimal $\tau_{c}^{\infty}$-generator $\mathsterling_{U}$ is a $\ast$-preserving derivation on $\mc{D}(M)$ 
then the statement follows.
\end{remark}
\begin{lemma}
\label{11011940}
Let $(M,U),(N,V)\in\mr{vf}^{\star}$ and $\phi\in\mr{Mor_{vf^{\star}}}((M,U),(N,V))$ thus 
$\phi^{\ast}\circ\exp_{M}^{U}(t)=\exp_{N}^{V}(t)\circ\phi^{\ast}$ and 
$\mathsterling_{U}\circ\exp_{M}^{U}(t)=\exp_{M}^{U}(t)\circ\mathsterling_{U}$ for every $t\in\R$.
\end{lemma}
\begin{proof}
Since $U$ and $V$ are $\phi$-related we have that 
$\phi^{\ast}\circ\mathsterling_{U}=\mathsterling_{V}\circ\phi^{\ast}$, thus the first equality follows since 
$\phi^{\ast}$ is $\tau_{c}^{\infty}$-continuous, the second equality follows since $\mathsterling_{U}$
is $\tau_{c}^{\infty}$-continuous.
\end{proof}
\begin{lemma}
\label{11231514}
Let $X$ be a Montel space, $Y$ a topological space, $U,V:Y\to\mf{L}_{s}(X)$ continuous at $t_{0}\in Y$ 
and such that $\{U(t)\,\vert\,t\in Y\}$ is equicontinuous. 
If the neighbourhood filter of $t_{0}$ in $Y$ admits a countable basis, then 
$Z^{U,V}:Y\to\mf{L}_{s}(\mf{L}_{b}(X))$ is continuous at $t_{0}$, 
where $Z^{U,V}:t\mapsto(T\mapsto U(t)\circ T\circ V(t))$.
\end{lemma}
\begin{proof}
In this proof we let $Z$ denote $Z^{U,V}$ which 
is well-defined namely $Z(t)\in\mf{L}(\mf{L}_{b}(X))$ for every $t\in Y$.
Indeed let $p$ be a continuous seminorm on $X$, $B$ a bounded subset of $X$ and $T\in\mf{L}(X)$, thus
$p_{B}(Z(t)T)=(q^{t})_{C^{t}}(T)$ with $q^{t}=p\circ U(t)$ and $C^{t}=V(t)B$. But $q^{t}$ is a
continuous seminorm of $X$, while $C^{t}$ is bounded in $X$ since $V(t)$ is linear and continuous, thus
$q^{t}_{C^{t}}$ is a continuous seminorm of $\mf{L}_{b}(X)$ and then $Z(t)\in\mf{L}(\mf{L}_{b}(X))$. 
Next assume that the neighbourhood filter of $t_{0}$ in $Y$ admits a countable basis.
Now $X$ is a Montel space thus it is sufficient to prove that for every sequence $\{t_{n}\}$ in $Y$ 
converging at $t_{0}$ and every $T\in\mf{L}(X)$, we have that $Z_{t_{n}}T$ converges at $T$ in $\mf{L}_{pc}(X)$.
Now since the equicontinuity hypothesis, there exists a continuous seminorm $q$ of $X$ such that 
for all $x\in X$ and $n\in\Z_{+}^{\ast}$ we have 
\begin{equation*}
\begin{aligned}
p(Z_{t_{n}}(T)x-Z_{t_{0}}(T)x)
&\leq 
p(U_{t_{n}}(TV_{t_{n}}x-TV_{t_{0}}x))+p((U_{t_{n}}-U_{t_{0}})TV_{t_{0}}x)
\\
&\leq
q(T(V_{t_{n}}x-V_{t_{0}}x))+p((U_{t_{n}}-U_{t_{0}})TV_{t_{0}}x);
\end{aligned}
\end{equation*}
so $p(Z_{t_{n}}(T)x-Z_{t_{0}}(T)x)$ converges at $0$, but $p$ is an arbitrary continuous seminorm on $X$, thus 
$Z_{t_{n}}T$ converges at $Z_{t_{0}}T$ in $\mf{L}_{s}(X)$. 
Finally $X$ is barrelled being Montel, thus by the Banach-Steinhaus Thm. 
we deduce that $Z_{t_{n}}T$ converges at $Z_{t_{0}}T$ in $\mf{L}_{pc}(X)$ which is what we claimed to prove.
\end{proof}
\begin{corollary}
\label{11231634}
Let $(M,U)\in\mr{vf}^{\star}$, thus $\Lambda_{M}^{U}$ is a $C_{0}$-group on $\mf{L}_{b}(\mc{D}(M))$. 
\end{corollary}
\begin{proof}
By letting 
$U_{+}=\ov{\exp}_{M}^{U}$, $V_{+}:\R_{+}\ni t\mapsto\exp(t)\exp_{M}^{U}(-t)$, 
$U_{-}=\underline{\exp}_{M}^{U}$ and $V_{-}:\R_{+}\ni t\mapsto\exp(t)\exp_{M}^{U}(t)$ 
the statement follows by Lemma \ref{11231514} and since $\Lambda_{M}^{U}\up\R_{+}=Z^{U_{+},V_{+}}$ and 
$\Lambda_{M}^{U}\up\R_{-}=Z^{U_{-},V_{-}}\circ\mf{i}$ where $\mf{i}:\R_{-}\ni\lambda\to-\lambda\in\R_{+}$.
\end{proof}
Since $\mc{D}(M)$ is dense in $\mc{H}_{g}$ we can set the following 
\begin{definition}
[The set $\ms{B}(\mc{M})$]
\label{10311558}
Let $\mc{M}=(M,g)$ be a semi-Riemannian manifold, define 
\begin{equation*}
\ms{B}(\mc{M})\coloneqq
\{
T\in\mf{L}(\mc{D}(M))\,\vert\,\mc{D}(M)\subseteq\mr{Dom}(T^{\intercal}), 
T^{\intercal}\mc{D}(M)\subseteq\mc{D}(M), T^{\dagger}\coloneqq 
T^{\intercal}\up\mc{D}(M)\in\mf{L}(\mc{D}(M))
\};
\end{equation*}
where $T^{\intercal}$ is the $\mc{H}_{g}-$adjoint of the operator $T$.
\end{definition}
\begin{remark}
\label{11251447}
Let $\mc{M}=(M,g)$ be a semi-Riemannian manifold, thus since the discussion in Notation 
we deduce that $\mr{DiffOp}^{k}(M)\subset\ms{B}(\mc{M})$ for every $k\in\Z_{+}^{\ast}$.
\end{remark}
\begin{proposition}
[$\ms{B}(\mc{M})$ is a $O^{\ast}$-algebra on $\mc{D}(M)$]
\label{11241301}
Let $\mc{M}=(M,g)$ be a semi-Riemannian manifold, thus $\ms{B}(\mc{M})$ is a $O^{\ast}$-algebra on $\mc{D}(M)$
in particular it is a unital $\ast$-algebra with involution $(\cdot)^{\dagger}$.
\end{proposition}
\begin{proof}
Let $S,T\in\ms{B}(\mc{M})$, thus $T$ is closable since $T^{\intercal}$ is densely defined.
Since $(S+T)^{\intercal}\supseteq S^{\intercal}+T^{\intercal}$ 
we obtain $\mc{D}(M)\subseteq\mr{Dom}((S+T)^{\intercal})$
and $(S+T)^{\dagger}=S^{\dagger}+T^{\dagger}\in\mf{L}(\mc{D}(M))$.
Next since $(ST)^{\intercal}\supseteq T^{\intercal}S^{\intercal}$ 
we obtain $\mc{D}(M)\subseteq\mr{Dom}((ST)^{\intercal})$
and $(ST)^{\dagger}=T^{\dagger}S^{\dagger}\in\mf{L}(\mc{D}(M))$.
Finally $\ov{T}=(T^{\intercal})^{\intercal}\subseteq (T^{\dagger})^{\intercal}$ since 
$T^{\dagger}\subseteq T^{\intercal}$, so $\mc{D}(M)\subseteq\mr{Dom}((T^{\dagger})^{\intercal})$ and 
$(T^{\dagger})^{\dagger}=T\in\mf{L}(\mc{D}(M))$.
\end{proof}
\begin{definition}
[The Categories $\mf{vf}_{0}$ and $\mf{vf}$]
\label{11241119}
Define $\mf{vf}_{0}$ to be the unique category whose object set consists of the couples $(\mc{M},U)$ where 
$\mc{M}=(M,g)$ is a semi-Riemannian manifold, $(M,U)\in\mr{vf}^{\star}$ and 
\begin{equation*}
\mu_{g}\circ\mathsterling_{U}=\ze.
\end{equation*}
The morphism set of $\mf{vf}_{0}$ is such that $\mr{Mor}_{\mf{vf}_{0}}((\mc{M},U),(\mc{N},V))$ consists of those
$\phi\in\mr{Mor_{vf^{\star}}}((M,U),(N,V))$ such that $\phi^{\ast}g=g^{\prime}$ where $\mc{N}=(N,g^{\prime})$,
and whose composition is given by $\phi\circ_{\mf{vf}_{0}}\psi=\psi\circ\phi$ with $\circ$ the map composition.
Let $\mf{vf}$ be the subcategory of $\mf{vf}_{0}$ with the same object set and
$\mr{Mor}_{\mf{vf}}((\mc{M},U),(\mc{N},V))$ consisting of those
$\phi\in\mr{Mor}_{\mf{vf}_{0}}((\mc{M},U),(\mc{N},V))$ such that $\phi$ is a diffeomorphism.
\end{definition}
\begin{definition}
\label{10311600}
Let $\mc{M}=(M,g)$ be a semi-Riemannian manifold, define 
\begin{equation*}
\begin{cases}
\omega_{f}:\ms{B}(\mc{M})\to\C,\,T\mapsto\lr{f}{Tf}_{\mc{H}_{g}},\,f\in\mc{D}(M);
\\
\omega_{f}^{Q}:\ms{B}(\mc{M})\to\C,\,T\mapsto\omega_{f}(Q^{\dagger}TQ),\,f\in\mc{D}(M), Q\in\ms{B}(\mc{M});
\\
F_{\mc{M}}\coloneqq\{\{\omega_{f}\,\vert\,f\in B\}\,\vert\,B\in\mr{Bounded}(\mc{D}(M))\}.
\end{cases}
\end{equation*}
\end{definition}
\begin{proposition}
\label{10311601}
Let $\mc{M}=(M,g)$ be a semi-Riemannian manifold, $B\in\mr{Bounded(\mc{D}(M))}$, $T,Q\in\ms{B}(\mc{M})$.
Thus $\{|\omega_{f}(T)|\,\vert\,f\in B\}$ is bounded and $\{\omega_{f}^{Q}\,\vert\,f\in B\}\in F_{\mc{M}}$. 
\end{proposition}
\begin{proof}
$T(B)$ is $\tau_{c}^{\infty}$-bounded since $T$ is $(\tau_{c}^{\infty},\tau_{c}^{\infty})$-continuous,
thus $B$ and $T(B)$ are $\|\cdot\|_{\mc{H}_{g}}$-bounded since $\mc{D}(M)\hookrightarrow\mc{H}_{g}$, 
so the first part of the statement follows since 
$\sup_{f\in B}|\lr{f}{Tf}_{\mc{H}_{g}}|\leq\sup_{f\in B}\|f\|_{\mc{H}_{g}}\sup_{f\in B}\|Tf\|_{\mc{H}_{g}}$.
The second part follows since $\omega_{f}^{Q}=\omega_{Qf}$ and $Q(B)$ is $\tau_{c}^{\infty}$-bounded. 
\end{proof}
The first part of Prp. \ref{10311601} allows us to give the following 
\begin{definition}
[The Topology $\tau_{\mc{M}}$ on $\ms{B}(\mc{M})$]
\label{14111822}
Let $\mc{M}=(M,g)$ be a semi-Riemannian manifold, define $\tau_{\mc{M}}$ to be the locally convex topology on
$\ms{B}(\mc{M})$ generated by the set of seminorms $\{q^{B}\,\vert\,B\in\mr{Bounded}(\mc{D}(M))\}$ where
$q^{B}:\ms{B}(\mc{M})\to\R_{+}$ $T\mapsto\sup_{f\in B}|\omega_{f}(T)|$.
\end{definition}
\begin{remark}
\label{11191746}
Let $\mc{M}=(M,g)$ be a semi-Riemannian manifold, thus by the polarization formula and since for any locally 
convex space $X$, $\lambda,\mu\in\C$ and $B,C$ bounded subsets of $X$, the set $\lambda B+\mu C$ is bounded 
in $X$ we deduce that the topology $\tau_{\mc{M}}$ is generated by the following set of seminorms 
$\{q^{B,C}\,\vert\,B,C\in\mr{Bounded}(\mc{D}(M))\}$, where 
$q^{B,C}(T)\coloneqq\sup_{(f,g)\in B\times C}|\lr{f}{Tg}_{\mc{H}_{g}}|$.
\end{remark}
\begin{theorem}
[The topological $\ast$-algebra $\mf{B}(\mc{M})$ and the morphism $\mc{T}(\phi)$]
\label{14111827}
Let $\mc{M}=(M,g)$ be a semi-Riemannian manifold, thus
\begin{enumerate}
\item
$\ms{B}(\mc{M})[\tau_{\mc{M}}]\in\mr{tsa}$
\label{14111827st1}
\item
if $\mc{N}=(N,g^{\prime})$ is a semi-Riemannian manifold and $\phi:N\to M$ is a smooth diffeomorphism
such that $\phi^{\ast}g=g^{\prime}$, then by letting
\begin{equation*}
\begin{cases}
\mc{T}(\phi):\ms{B}(\mc{M})\to\ms{B}(\mc{N}),
\\
T\mapsto\phi^{\ast}\circ T\circ(\phi^{-1})^{\ast};
\end{cases}
\end{equation*}
we have that 
\begin{equation*}
\mc{T}(\phi)\in\mr{Mor_{tsa}}(\ms{B}(\mc{M})[\tau_{\mc{M}}],\ms{B}(\mc{N})[\tau_{\mc{N}}]).
\end{equation*}
\label{14111827st2}
\end{enumerate}
\end{theorem}
\begin{proof}
St. \eqref{14111827st1} follows by Prp. \ref{10311601} and \cite[Lemma 1.5.7]{schm} 
applied to the unital $\ast$-algebra $\ms{B}(\mc{M})$ and the set $F_{\mc{M}}$. 
Next let $T\in\ms{B}(\mc{M})$ thus we have what 
follows. $\mc{T}(\phi)(T)\in\mf{L}(\mc{D}(N))$ since $\phi^{\ast}$, $T$ and $(\phi^{-1})^{\ast}$ are all linear 
and $\tau_{c}^{\infty}$-contiuous operators. Next by recalling the property 
$(\phi^{\ast})^{\intercal}=(\phi^{-1})^{\ast}$ discussed in Notation, we have
\begin{equation*}
\begin{aligned}
\mc{T}(\phi)(T)^{\intercal}
&=
((\phi^{-1})^{\ast})^{\intercal}\circ T^{\intercal}\circ(\phi^{\ast})^{\intercal}
\\
&=
\phi^{\ast}\circ T^{\intercal}\circ(\phi^{-1})^{\ast}.
\end{aligned}
\end{equation*}
Thus $\mc{D}(N)\subseteq\mr{Dom}(\mc{T}(\phi)(T)^{\intercal})$, 
$\mc{T}(\phi)(T)^{\intercal}\mc{D}(N)\subseteq\mc{D}(N)$ and 
$\mc{T}(\phi)(T)^{\dagger}=\mc{T}(\phi)(T^{\dagger})\in\mf{L}(\mc{D}(N))$.
Therefore $\mc{T}(\phi)$ is well-set namely $\mc{T}(\phi)(T)\in\ms{B}(\mc{N})$, moreover 
$\mc{T}(\phi)$ is a $\ast$-morphism.
Finally let us prove the continuity of $\mc{T}(\phi)$. For every $f\in\mc{D}(N)$ we have that 
\begin{equation*}
\begin{aligned}
\lr{f}{\mc{T}(\phi)(T)f}_{\mc{H}_{g^{\prime}}}
&=
\lr{(\phi^{\ast})^{\intercal}f}{T((\phi^{-1})^{\ast}f)}_{\mc{H}_{g}}
\\
&=
\lr{(\phi^{-1})^{\ast}f}{T((\phi^{-1})^{\ast}f)}_{\mc{H}_{g}};
\end{aligned}
\end{equation*}
but $(\phi^{-1})^{\ast}$ is $(\tau_{c}^{\infty}(N),\tau_{c}^{\infty}(M))$-continuous, therefore 
$(\phi^{-1})^{\ast}(B)$ is $\tau_{c}^{\infty}(M)$-bounded for every $\tau_{c}^{\infty}(N)$-bounded set $B$
hence $\mc{T}(\phi)$ is $(\tau_{\mc{M}},\tau_{\mc{N}})$-continuous and st. \eqref{14111827st2} follows.
\end{proof}
The above result justifies the following
\begin{definition}
\label{11270809}
Let $\mc{M}=(M,g)$ be a semi-Riemannian manifold, define $\mf{B}(\mc{M})$ to be the unital topological 
$\ast$-algebra $\ms{B}(\mc{M})[\tau_{\mc{M}}]$.
\end{definition}
\begin{proposition}
\label{11271148}
Let $\mc{M}=(M,g)$ be a semi-Riemannian manifold, thus 
$\ms{B}(\mc{M})\subset\mc{L}(\mc{D}(M)_{\ms{B}(\mc{M})},\mc{D}(M)_{\ms{B}(\mc{M})}^{+})$ and 
$\ms{B}(\mc{M})[\tau_{b}]$ is a unital topological $\ast$-algebra such that 
$\ms{B}(\mc{M})[\tau_{b}]\hookrightarrow\mf{B}(\mc{M})$.
\end{proposition}
\begin{proof}
Let $T\in\ms{B}(\mc{M})$, thus $(f,h)\mapsto\lr{Tf}{h}_{\mc{H}_{g}}$ is clearly jointly continuous w.r.t. 
the graph topology of $\ms{B}(\mc{M})$ on $\mc{D}(M)$ so the inclusion in the statement follows, 
in particular $\ms{B}(\mc{M})[\tau_{b}]$ is well-set and it is a topological $\ast$-algebra since 
\cite[Prp. 3.3.10]{schm}. 
Next it is well-known that the graph topology of a $O^{\ast}$-algebra $\mc{A}$ on a dense subspace $D$ of a 
Hilbert space $\mc{H}$ is the weakest among all the locally convex topologies $\tau$ on $D$ such that 
$\mc{A}\subset\mf{L}(D[\tau],\mc{H})$. But $\ms{B}(\mc{M})\subset\mf{L}(\mc{D}(M),\mc{H}_{g})$, therefore 
$\mc{D}(M)\hookrightarrow\mc{D}(M)_{\ms{B}(\mc{M})}$. 
In particular $\mr{Bounded}(\mc{D}(M))\subseteq\mr{Bounded}(\mc{D}(M)_{\ms{B}(\mc{M})})$ which implies 
$\ms{B}(\mc{M})[\tau_{b}]\hookrightarrow\mf{B}(\mc{M})$.
\end{proof}
\begin{lemma}
\label{11240851}
Let $(\mc{M},U)\in\mf{vf}_{0}$ with $\mc{M}=(M,g)$ and $t\in\R$, then 
$\exp_{M}^{U}(t)$ extends uniquely
to a unitary operator $\mf{exp}_{\mc{M}}^{U}(t)$ on $\mc{H}_{g}$ such that 
$\mf{exp}_{\mc{M}}^{U}(t)^{\intercal}=\mf{exp}_{\mc{M}}^{U}(-t)$.
\end{lemma}
\begin{proof}
$\mc{D}(M)\hookrightarrow\mc{K}(M)$ and $\mu_{g}\circ\mathsterling_{U}=\ze$ imply that 
$\mu_{g}\circ\exp_{M}^{U}(t)=\mu_{g}$ for all $t\in\R$, therefore for every $f,h\in\mc{D}(M)$ we have 
\begin{equation}
\label{11240916}
\begin{aligned}
\lr{\exp_{M}^{U}(t)f}{\exp_{M}^{U}(t)h}_{\mc{H}_{g}}
&=
\mu_{g}(\ov{\exp_{M}^{U}(t)f}\exp_{M}^{U}(t)h)
\\
&=
(\mu_{g}\circ\exp_{M}^{U}(t))(\ov{f}h)
\\
&=
\mu_{g}(\ov{f}h)
\\
&=\lr{f}{h}_{\mc{H}_{g}};
\end{aligned}
\end{equation}
where the second equality follows since Rmk. \ref{11240847}. 
Thus $\exp_{M}^{U}(t)$ is a unitary operator on the dense subspace $\mc{D}(M)$ of $\mc{H}_{g}$, 
which then extends uniquely to a unitary operator $\mf{exp}_{\mc{M}}^{U}(t)$ on $\mc{H}_{g}$.
Next since $\mf{exp}_{\mc{M}}^{U}(t)$ is unitarity and since $\exp_{M}^{U}(t)^{-1}=\exp_{M}^{U}(-t)$, 
we have that $\mf{exp}_{\mc{M}}^{U}(t)^{\intercal}\up\mc{D}(M)=\mf{exp}_{\mc{M}}^{U}(-t)\up\mc{D}(M)$
and the equality in the statement follows.
\end{proof}
\begin{corollary}
\label{11261248}
Let $(\mc{M},U)\in\mf{vf}_{0}$ with $\mc{M}=(M,g)$, then there exists 
a unique $C_{0}$-group $\mf{exp}_{\mc{M}}^{U}$ on $\mc{H}_{g}$ of unitary operators 
extending $\exp_{M}^{U}$ and 
whose infinitesimal generator $\mf{l}_{U}$ extends $\mathsterling_{U}$.
\end{corollary}
\begin{proof}
Since Lemma \ref{11240851} it remains only to prove the $C_{0}$-property and 
$\mf{l}_{U}\supseteq\mathsterling_{U}$. To this end let $f\in\mc{D}(M)$,
thus $t\mapsto\mf{exp}_{\mc{M}}^{U}(t)f$ is $\|\cdot\|_{\mc{H}_{g}}$-continuous
since $t\mapsto\exp_{M}^{U}(t)f$ is $\tau_{c}^{\infty}$-continuous by construction and since 
$\mc{D}(M)\hookrightarrow\mc{H}_{g}$. Next $\mf{exp}_{\mc{M}}^{U}(\R)$ is 
$(\|\cdot\|_{\mc{H}_{g}},\|\cdot\|_{\mc{H}_{g}})$-equicontinuous since isometric, while $\mc{D}(M)$ is dense 
in $\mc{H}_{g}$. Therefore since over equicontinuous sets the uniform structure of simple convergence coincides 
with the uniform structure of simple convergence over a total set, we conclude that for every $h\in\mc{H}_{g}$ 
the map $t\mapsto\mf{exp}_{\mc{M}}^{U}(t)h$ is $\|\cdot\|_{\mc{H}_{g}}$-continuous namely 
$\mf{exp}_{\mc{M}}^{U}$ is a $C_{0}$-group on $\mc{H}_{g}$. 
Finally $\mf{l}_{U}$ extends  the infinitesimal $\tau_{c}^{\infty}$-generator $\mathsterling_{U}$ of 
$\exp_{M}^{U}$ since $\mf{exp}_{\mc{M}}^{U}$ extends $\exp_{M}^{U}$ and since 
$\mc{D}(M)\hookrightarrow\mc{H}_{g}$.
\end{proof}
\begin{corollary}
\label{11241005}
Let $(\mc{M},U)\in\mf{vf}_{0}$ with $\mc{M}=(M,g)$ and $t\in\R$, then 
$\exp_{M}^{U}(t)\in\mf{B}(\mc{M})$ such that $\exp_{M}^{U}(t)^{\dagger}=\exp_{M}^{U}(-t)$, in particular
$\Lambda_{M}^{U}(t)\mf{B}(\mc{M})\subseteq\mf{B}(\mc{M})$.
\end{corollary}
\begin{proof}
Since Cor. \ref{11261248}.
\end{proof}
Cor. \ref{11241005} allows to give the following 
\begin{definition}
[The Group $\Gamma_{\mc{M}}^{U}$]
\label{11241536}
Let $(\mc{M},U)\in\mf{vf}_{0}$, define $\Gamma_{\mc{M}}^{U}:\R\to\mr{End_{vct}}(\mf{B}(\mc{M}))$ 
such that 
\begin{equation*}
\Gamma_{\mc{M}}^{U}:t\mapsto(T\mapsto\Lambda_{M}^{U}(t)(T)),
\end{equation*}
where $M$ is the manifold underlying $\mc{M}$.
\end{definition}
\begin{lemma}
\label{11250550} 
Let $\mc{M}$ be a semi-Riemannian manifold, thus the topology $\tau_{\mc{M}}$ is generated 
by a collection of seminorms extending to $\mf{L}_{b}(\mc{D}(M))$-continuous seminorms.
\end{lemma}
\begin{proof}
Let $B$ and $C$ be bounded subsets of $\mc{D}(M)$ and let 
$\zeta^{B}:\mc{D}(M)\to\R_{+}$ $h\mapsto\sup_{f\in B}|\lr{f}{h}_{\mc{H}_{g}}|$ 
finite since $\mc{D}(M)\hookrightarrow\mc{H}_{g}$. 
Now since $\mc{D}(M)\hookrightarrow\mc{K}(M)$ and $h\cdot\mu_{g}\in\mc{K}(M)^{\prime}$ for every $h\in\mc{K}(M)$
we have $\lr{f}{\cdot}_{\mc{H}_{g}}\circ\imath_{\mc{D}(M)}^{\mc{H}_{g}}=
(\ov{f}\cdot\mu_{g})\circ\imath_{\mc{D}(M)}^{\mc{K}(M)}\in\mc{D}(M)^{\prime}$. 
Therefore $\zeta^{B}$ is lower $\tau_{c}^{\infty}$-semicontinuous since superior envelop of 
$\tau_{c}^{\infty}$-continuous maps, hence $\zeta^{B}$ is $\tau_{c}^{\infty}$-continuous 
since $\mc{D}(M)$ is barrelled.
Thus $(\zeta^{B})_{C}$ is a continuous seminorm of $\mf{L}_{b}(\mc{D}(M))$ 
so the statement follows by the fact that $(\zeta^{B})_{C}=q^{B,C}$ and by Rmk. \ref{11191746}.
\end{proof}
\begin{definition}
Let $\mc{M}$ be a semi-Riemannian manifold, define 
$\mf{U}(M)\coloneqq\{T\in\mf{B}(\mc{M})\,\vert\,T^{\dagger}=T^{-1}\}$ 
endowed with the relative topology of $\mf{B}(\mc{M})$ and with the group structure inherited by the 
multiplication on $\mf{B}(\mc{M})$. 
\end{definition}
Notice that in general $\mf{U}(M)$ needs not to be a topological group.
\begin{theorem}
[$\Gamma_{\mc{M}}^{U}$ is a $C_{0}$-group on $\mf{B}(\mc{M})$ of $\ast-$automorphisms]
\label{11191752}
Let $(\mc{M},U),(\mc{N},V)\in\mf{vf}$ and $\phi\in\mr{Mor}_{\mf{vf}}((\mc{M},U),(\mc{N},V))$, thus 
by letting $M$ be the manifold underlying $\mc{M}$, we have 
\begin{enumerate}
\item
$\exp_{M}^{U}\in\mc{C}(\R,\mf{U}(\mc{M}))$ morphism of groups;
\label{11191752st1}
\item
$\Gamma_{\mc{M}}^{U}$ is a $C_{0}$-group on $\mf{B}(\mc{M})$ of $\ast-$automorphisms;
\label{11191752st2}
\item
$\mc{T}(\phi)\circ\Gamma_{\mc{M}}^{U}(t)=\Gamma_{\mc{N}}^{V}(t)\circ\mc{T}(\phi)$, for every $t\in\R$.
\label{11191752st3}
\end{enumerate}
\end{theorem}
\begin{proof}
$\exp_{M}^{U}$ is a morphism of the groups involved in the statement since Cor. \ref{11241005}.
Let us prove the continuity. Since $\exp_{M}^{U}$ is a $C_{0}$-group on $\mc{D}(M)$ by construction
and since $\mc{D}(M)$ is barrelled it follows by the Banach-Steinhaus Thm. that 
$\exp_{M}^{U}\in\mc{C}(\R_{+},\mf{L}_{pc}(\mc{D}(M)))$, 
then $\exp_{M}^{U}\in\mc{C}(\R_{+},\mf{L}_{b}(\mc{D}(M)))$ since $\mc{D}(M)$ is a Montel space, 
therefore st. \eqref{11191752st1} follows by Lemma \ref{11250550}.
Next let $B$ be a $\tau_{c}^{\infty}$-bounded set and $t\in\R$, 
thus since $\exp_{M}^{U}(t)^{\dagger}=\exp_{M}^{U}(-t)$ by Cor. \ref{11241005},
we have $q^{B}\circ\Gamma_{\mc{M}}^{U}(t)=q^{\exp_{M}^{U}(-t)B}$, 
but $\exp_{M}^{U}(-t)B$ is $\tau_{c}^{\infty}$-bounded since $\exp_{M}^{U}(-t)$ is 
$\tau_{c}^{\infty}$-continuous, therefore $\Gamma_{\mc{M}}^{U}(t)\in\mf{L}(\mf{B}(\mc{M}))$.
Thus $\Gamma_{\mc{M}}^{U}$ is a $C_{0}$-group on $\mf{B}(\mc{M})$ 
since Cor. \ref{11231634} and Lemma \ref{11250550}.
Finally $\Gamma_{\mc{M}}^{U}(t)$ is a $\ast$-automorphism of $\mf{B}(\mc{M})$ since Cor. \ref{11241005},
so st. \eqref{11191752st2} is proven. St. \ref{11191752st3} follows since Lemma \ref{11011940}.
\end{proof}
Next we set the following definition of $\dagger$ here used instead of the 
analog one given in \cite[Def. 1.4.14]{sil2cat}. Clearly the corresponding of \cite[Cor. 1.4.16]{sil2cat} holds.
\begin{definition}
\label{12051824}
Let $D$ be a category, $\mr{a},\mr{b}\in\mr{Fct}(D,\mr{tls})$ and 
$T\in\prod_{d\in D}\mr{Mor}_{\mr{tls}}(\mr{a}(d),\mr{b}(d))$, then
define $T^{\dagger}\in\prod_{d\in D}\mr{Mor}_{\mr{tls}}(\mr{b}(d)^{\prime},\mr{a}(d)^{\prime})$ 
such that $T^{\dagger}(e)\coloneqq(T(e))^{\dagger}$ for all $e\in D$, where 
$S^{\dagger}(\omega)\coloneqq\omega\circ S$.
\end{definition}
We conclude this section with the existence of the category $\mf{Chdv}_{0}$ 
obtained by forgetting in the category $\mf{Chdv}$ uniquely determined in \cite[Cor. 1.4.18]{sil2cat} 
the category $\mr{ptls}$ into the category $\mr{tls}$ and the category $\mr{ptsa}$ into the category $\mr{tls}$. 
Then we obtain the corresponding functor $\ps{\Uppsi}_{0}$ 
in analogy with the functor $\ps{\Uppsi}$ in \cite[Thm. 1.4.19]{sil2cat} 
Before the next result let us recall that for any $\mf{A}\in\mf{dp}$, $\mf{A}^{\dagger}$ is defined in 
\cite[Def. 1.4.13]{sil2cat} and that since \cite[Thm. 1.4.15]{sil2cat} $\mf{r}\circ\upsigma_{\mf{A}}^{\dagger}$ 
is well-set. 
\begin{proposition}
[The category $\mf{Chdv}_{0}$]
\label{12131143}
There exists a unique category $\mf{Chdv}_{0}$ whose object set equals the object set of $\mf{dp}$ and 
whose morphism set is such that for every $\mf{A},\mf{B},\mf{C}\in\mf{Chdv}_{0}$ we have 
\begin{equation}
\label{123121370}
\begin{split}
\mr{Mor}_{\mf{Chdv}_{0}}&(\mf{A},\mf{B})=
\\
&\coprod_{f\in\mr{Fct}_{\mr{top}}(\mr{G}_{\mf{B}},\mr{G}_{\mf{A}})}
\mr{Mor}_{\mr{Fct}(\mr{G}_{\mf{B}}^{op},\mr{tls})}
(\mf{r}\circ\upsigma_{\mf{B}}^{\dagger},\mf{r}\circ\upsigma_{\mf{A}}^{\dagger}\circ f)
\times
\mr{Mor}_{\mr{Fct}(\mr{G}_{\mf{B}},\mr{tls})}
(\mf{q}_{0}\circ\upsigma_{\mf{A}}\circ f,\mf{q}_{0}\circ\upsigma_{\mf{B}})
\end{split}
\end{equation}
and 
\begin{equation}
\label{123121130}
\begin{aligned}
(\circ):\mr{Mor}_{\mf{Chdv}_{0}}(\mf{B},\mf{C})&\times 
\mr{Mor}_{\mf{Chdv}_{0}}(\mf{A},\mf{B})\to\mr{Mor}_{\mf{Chdv}_{0}}(\mf{A},\mf{C}),
\\
(g,L,S)&\circ(f,H,T)
\coloneqq
(f\circ g,(H\ast\un_{g})\circ L,S\circ(T\ast \un_{g})).
\end{aligned}
\end{equation}
Moreover the maps $\mf{A}\mapsto\mf{A}$ and 
$(f,T)\mapsto(f,(\un_{\mf{q}_{0}}\ast T)^{\dagger},\un_{\mf{q}_{0}}\ast T)$ 
determine uniquely an element $\ps{\Uppsi}_{0}\in\mr{Fct}(\mf{dp},\mf{Chdv}_{0})$.
\end{proposition}
Next in analogy with \cite[Def. 1.5.8]{sil2cat} we give the following 
\begin{definition}
[The fibered category of $0-$species]
\label{12041523}
Define $\mf{Sp}_{0}$ the fibered category over $2-\mf{dp}$ such that for all $\mf{D}\in 2-\mf{dp}$
\begin{equation*}
\mf{Sp}_{0}(\mf{D})=2-\mf{dp}(\mf{D},\mf{Chdv}_{0}),
\end{equation*}
moreover set
\begin{equation*}
\mf{Sp}_{\ast}^{0}\coloneqq\{(\mr{a},\mr{b})\in\mf{Sp}_{0}\times\mf{Sp}_{0}
\,\vert\, 
d(\mr{a})=d(\mr{b})
\}.
\end{equation*}
\end{definition}
\section{Construction of the classical and quantum $0$-species $\ms{x}$ and $\ms{z}$} 
\label{12111459}
Since $\exp_{M}^{U}$ is a $C_{0}$-group on $\mc{D}(M)$ we immediatedly obtain the following 
\begin{proposition}
\label{11271635}
Let $(\mc{M},U)\in\mf{vf}_{0}$. Thus there exists a unique 
$\lr{\lr{\mc{M}}{U}}{\mr{F}_{\lr{\mc{M}}{U}}}\in\mf{dp}$ with the following properties.
$\lr{\mc{M}}{U}$ is the unique $\mr{top}$-quasi enriched category with the following properties. 
The object set of $\lr{\mc{M}}{U}$ is the set of all subsets of $M$, 
the morphism set of $\lr{\mc{M}}{U}$ is such that for every $X,Y\in\lr{\mc{M}}{U}$ we have 
$\mr{Mor}_{\lr{\mc{M}}{U}}(X,Y)=\{(X,Y)\}\times\mr{mor}_{\lr{\mc{M}}{U}}(X,Y)$, with
\begin{equation*}
\mr{mor}_{\lr{\mc{M}}{U}}(X,Y)=\{t\in\R\,\vert\,\exp_{M}^{U}(t)\mc{D}(M,X)=\mc{D}(M,Y)\},
\end{equation*}
where we let $M$ be the manifold underlying $\mc{M}$ and $\mc{D}(M,X)$ be the topological sub $\ast$-algebra 
of $\mc{D}(M)$ of those maps whose support is contained in $X$. 
The topology on $\mr{Mor}_{\lr{\mc{M}}{U}}(X,Y)$ is that induced by the topology on $\R$,
while the composition is that inherited by the addition in $\R$. 
While $\mr{F}_{\lr{\mc{M}}{U}}\in\mr{Fct_{top}}(\lr{\mc{M}}{U},\mr{tsa})$ such that for every 
$t\in\mr{mor}_{\lr{\mc{M}}{U}}(X,Y)$ we have 
\begin{equation*}
\begin{cases}
\mr{F}_{\lr{\mc{M}}{U}}(X)=\mc{D}_{1}(M,X),
\\
\mr{F}_{\lr{\mc{M}}{U}}((X,Y),t):\mc{D}_{1}(M,X)\to\mc{D}_{1}(M,Y),
\\
(f,\lambda)\mapsto(\exp_{M}^{U}(t)f,\lambda);
\end{cases}
\end{equation*}
with $\mc{D}_{1}(M,X)\in\mr{tsa}$ the unitization of $\mc{D}(M,X)$.
\end{proposition}
\begin{theorem}
\label{11280957}
There exists a unique $\mf{x}\in\mr{Fct}(\mf{vf},\mf{dp})$ 
such that for all $(\mc{M},U),(\mc{N},V)\in\mf{vf}$ and 
$\phi\in\mr{Mor}_{\mf{vf}}((\mc{M},U),(\mc{N},V))$
\begin{enumerate}
\item
$\mf{x}((\mc{M},U))=\lr{\lr{\mc{M}}{U}}{\mr{F}_{\lr{\mc{M}}{U}}}$,
\label{11280957st1}
\item
$\mf{x}(\phi)=(f^{\phi},\ms{T}_{1}^{\phi})$;
\label{11280957st2}
\end{enumerate}
where 
$f^{\phi}\in\mr{Fct_{top}}(\lr{\mc{N}}{V},\lr{\mc{M}}{U})$ 
and 
\begin{equation}
\label{11281040}
\ms{T}_{1}^{\phi}\in\mr{Mor}_{\mr{Fct}(\lr{\mc{N}}{V},\mr{tsa})}
(\mr{F}_{\lr{\mc{M}}{U}}\circ f^{\phi},\mr{F}_{\lr{\mc{N}}{V}});
\end{equation}
such that for all $Y,Z\in\lr{\mc{N}}{V}$ and $t\in\mr{mor}_{\lr{\mc{N}}{V}}(Y,Z)$
\begin{enumerate}
\item
$f^{\phi}(Y)=\phi(Y)$;
\label{11280957st3}
\item
$f^{\phi}((Y,Z),t)=((\phi(Y),\phi(Z)),t)$;
\label{11280957st4}
\item
$\ms{T}^{\phi}(Y):\mc{D}(M,\phi(Y))\to\mc{D}(N,Y)\quad h\mapsto\phi^{\ast}h$;
\label{11280957st5}
\item
$\ms{T}_{1}^{\phi}(Y):\mc{D}_{1}(M,\phi(Y))\to\mc{D}_{1}(N,Y)\quad(h,\lambda)\mapsto(\phi^{\ast}h,\lambda)$;
\label{11280957st6}
\end{enumerate}
where $M$ and $N$ are the manifolds underlying $\mc{M}$ and $\mc{N}$ respectively. In particular 
$\ps{\Uppsi}\circ\mf{x}\in\mf{Sp}(\mf{vf})$ and $\ps{\Uppsi}_{0}\circ\mf{x}\in\mf{Sp}_{0}(\mf{vf})$.
\end{theorem}
\begin{proof}
Let us take the properties of the statement as definition of $\mf{x}$. 
Let $t\in\mr{mor}_{\lr{\mc{N}}{V}}(Y,Z)$ and 
$f\in\mc{D}(M,\phi(Y))$, so since Lemma \ref{11011940} we have 
\begin{equation*}
\begin{aligned}
\phi^{\ast}(\exp_{M}^{U}(t)f)
&=
\exp_{N}^{V}(t)(\phi^{\ast}f)
\\
&\in\exp_{N}^{V}(t)(\mc{D}(N,Y))
\subseteq
\mc{D}(N,Z);
\end{aligned}
\end{equation*}
namely 
\begin{equation*}
\begin{aligned}
\exp_{M}^{U}(t)f\in
(\phi^{-1})^{\ast}\mc{D}(N,Z)
&=
(\phi^{-1})^{\ast}\mc{D}(N,\phi^{-1}(\phi(Z)))
\\
&\subseteq\mc{D}(M,\phi(Z)).
\end{aligned}
\end{equation*}
Therefore $f^{\phi}((Y,Z),t)\in\mr{Mor}_{\lr{\mc{M}}{U}}(\phi(Y),\phi(Z))$, next $f^{\phi}$ is clearly 
continuous and composition preserving so $f^{\phi}\in\mr{Fct_{top}}(\lr{\mc{N}}{V},\lr{\mc{M}}{U})$.
Next $\ms{T}_{1}^{\phi}(Y)$ is continuous since $\phi^{\ast}$ is 
$(\tau_{c}^{\infty}(M),\tau_{c}^{\infty}(N))$-continuous,
moreover for every $f\in\mc{D}(M,\phi(Y))$ and $\lambda\in\C$, since Lemma \ref{11011940} we have 
\begin{equation*}
\begin{aligned}
(\ms{T}_{1}^{\phi}(Z)\circ\mr{F}_{\lr{\mc{M}}{U}}(\phi(Y),\phi(Z),t))(f,\lambda)
&=
((\phi^{\ast}\circ\exp_{M}^{U}(t))f,\lambda)
\\
&=
((\exp_{N}^{V}(t)\circ\phi^{\ast})f,\lambda)
\\
&=
(\mr{F}_{\lr{\mc{N}}{V}}(Y,Z,t)\circ\ms{T}_{1}^{\phi}(Y))(f,\lambda);
\end{aligned}
\end{equation*}
which proves \eqref{11281040}. 
Finally $\mf{x}(\psi\circ_{\mf{vf}}\phi)=\mf{x}(\psi)\circ_{\mf{dp}}\mf{x}(\phi)$ follows by the same 
line of reasoning we use in \cite[Thm. 1.6.24]{sil2cat} to prove that 
$\mf{a}(\psi\circ_{\ms{vf}_{0}}\phi)=\mf{a}(\psi)\circ_{\mf{dp}}\mf{a}(\phi)$.
\end{proof}
\begin{theorem}
\label{11281921}
Let $(\mc{M},U)\in\mf{vf}_{0}$, 
thus there exists a unique $\lr{\lr{\mc{M}}{U}}{\mc{F}_{\lr{\mc{M}}{U}}}\in\mf{dp}$
with the following properties. $\mc{F}_{\lr{\mc{M}}{U}}\in\mr{Fct_{top}}(\lr{\mc{M}}{U},\mr{tsa})$ 
such that for every subset $X$ and $Y$ of $M$ and every $t\in\mr{mor}_{\lr{\mc{M}}{U}}(X,Y)$ we have 
\begin{equation*}
\begin{cases}
\mc{F}_{\lr{\mc{M}}{U}}(X)=\mf{B}(\mc{M},X),
\\
\mc{F}_{\lr{\mc{M}}{U}}((X,Y),t):\mf{B}(\mc{M},X)\to\mf{B}(\mc{M},Y)\quad T\mapsto\Gamma_{\mc{M}}^{U}(t)T;
\end{cases}
\end{equation*}
where we let $M$ be the manifold underlying $\mc{M}$ and $\mf{B}(\mc{M},X)$ be the topological unital sub 
$\ast$-algebra of $\mf{B}(\mc{M})$ of those $T$ such that $T\mc{D}(M,X)\subseteq\mc{D}(M,X)$ and 
$T^{\dagger}\mc{D}(M,X)\subseteq\mc{D}(M,X)$.
\end{theorem}
\begin{proof}
Since Thm. \ref{11191752}\eqref{11191752st2}.
\end{proof}
\begin{theorem}
\label{11290640}
There exists a unique $\mf{z}\in\mr{Fct}(\mf{vf},\mf{dp})$ 
such that for all $(\mc{M},U),(\mc{N},V)\in\mf{vf}$ and $\phi\in\mr{Mor}_{\mf{vf}}((\mc{M},U),(\mc{N},V))$
\begin{enumerate}
\item
$\mf{z}((\mc{M},U))=\lr{\lr{\mc{M}}{U}}{\mc{F}_{\lr{\mc{M}}{U}}}$,
\item
$\mf{z}(\phi)=(f^{\phi},\mf{T}^{\phi})$;
\end{enumerate}
where 
\begin{equation}
\label{11290656}
\mf{T}^{\phi}\in\mr{Mor}_{\mr{Fct}(\lr{\mc{N}}{V},\mr{tsa})}
(\mc{F}_{\lr{\mc{M}}{U}}\circ f^{\phi},\mc{F}_{\lr{\mc{N}}{V}});
\end{equation}
such that for all $Y,Z\in\lr{\mc{N}}{V}$ we have 
\begin{equation*}
\mf{T}^{\phi}(Y):\mf{B}(\mc{M},\phi(Y))\to\mf{B}(\mc{N},Y)\quad T\mapsto\mc{T}(\phi)T.
\end{equation*}
In particular $\ps{\Uppsi}\circ\mf{z}\in\mf{Sp}(\mf{vf})$ and 
$\ps{\Uppsi}_{0}\circ\mf{z}\in\mf{Sp}_{0}(\mf{vf})$.
\end{theorem}
\begin{proof}
Let us take the properties of the statement as definition of $\mf{z}$. 
Let $Q\in\mf{B}(\mc{M},\phi(Y))$ thus since Thm. \ref{11280957}\eqref{11280957st5} we obtain 
\begin{equation*}
\mf{T}^{\phi}(Y)(Q)\up\mc{D}(N,Y)=\mr{T}^{\phi}(Y)\circ Q\circ\mr{T}^{\phi^{-1}}(\phi(Y)),
\end{equation*}
moreover $\mc{T}(\phi)\mf{B}(\mc{M})\subseteq\mf{B}(\mc{N})$ since Thm. \ref{14111827}\eqref{14111827st2},
thus we obtain that $\mf{T}^{\phi}(Y)\mf{B}(\mc{M},\phi(Y))\subseteq\mf{B}(\mc{N},Y)$ and then 
$\mf{T}^{\phi}$ is well-set. Next $\mf{T}^{\phi}(Y)$ is continuous since it is so $\mc{T}(\phi)$
according to Thm. \ref{14111827}\eqref{14111827st2}. Next for every $t\in\mr{mor}_{\lr{\mc{N}}{V}}(Y,Z)$ 
we have by Thm. \ref{11191752}\eqref{11191752st3}
\begin{equation*}
\begin{aligned}
(\mf{T}^{\phi}(Z)\circ\mc{F}_{\lr{\mc{M}}{U}}(\phi(Y),\phi(Z),t))Q
&=
(\mc{T}(\phi)\circ\Gamma_{\mc{M}}^{U}(t))Q
\\
&=
(\Gamma_{\mc{N}}^{V}(t)\circ\mc{T}(\phi))Q
\\
&=
(\mc{F}_{\lr{\mc{N}}{V}}(Y,Z,t)\circ\mf{T}^{\phi}(Y))Q;
\end{aligned}
\end{equation*}
which proves \eqref{11290656}. 
Finally for every $\psi\in\mr{Mor}_{\mf{vf}}$ which is $\mf{vf}-$composable to the left with $\phi$
we have $\mf{z}(\psi\circ_{\mf{vf}}\phi)=\mf{z}(\psi)\circ_{\mf{dp}}\mf{z}(\phi)$ since 
$\mc{T}(\phi\circ\psi)=\mc{T}(\psi)\circ\mc{T}(\phi)$. 
\end{proof}
\section{Construction of the natural transformation $\mf{J}$ from $\ms{x}$ to $\ms{z}$}
\label{12111615}
\begin{theorem}
\label{11301607}
Let $\mc{M}$ be a semi-Riemannian manifold with underlying manifold $M$, 
thus for every $f,h\in\mc{D}(M)$ we have 
\begin{equation}
\label{11301607a}
\mathsterling_{\mr{grad}_{\mc{M}}(f)}(h)=\mathsterling_{\mr{grad}_{\mc{M}}(h)}(f).
\end{equation}
Next let $U$ be such that $(M,U)\in\mr{vf}^{\star}$.
If there exists a frame $\{E_{i}\}_{i=1}^{n=\mr{dim}M}$ 
of orthonormal fields of $\mc{M}$ such that $[U,E_{i}]=\ze$ for every $i\in[1,n]\cap\Z$, 
then for every $f\in\mc{D}(M)$ and $t\in\R$ we have 
\begin{equation}
\label{11301607b}
\exp_{M}^{U}(t)\circ\mathsterling_{[\mr{grad}_{\mc{M}}(f),U]}
=
\mathsterling_{[(\mr{grad}_{\mc{M}}\circ\exp_{M}^{U}(t))(f),U]}\circ\exp_{M}^{U}(t).
\end{equation}
\end{theorem}
\begin{proof}
By the same symbol $[,]$ we shall denote the Lie braket of vector fields on $M$ and the Lie braket induced
by the associative product on $\mf{L}(\mc{D}(M))$, namely $[T,Q]=T\circ Q-Q\circ T$ for 
$T,Q\in\mf{L}(\mc{D}(M))$. Next let $\{E_{i}\}_{i=1}^{n=\mr{dim}M}$ be a frame of orthonormal fields 
of $\mc{M}$, let $\ep_{i}=\lr{E_{i}}{E_{i}}_{\mc{M}}$ for every $i\in[1,n]\cap\Z$. 
Thus for every vector field $W$ we have 
$\mathsterling_{W}=\sum_{i=1}^{n}\ep_{i}\lr{W}{E_{i}}_{\mc{M}}\cdot\mathsterling_{E_{i}}$. 
Next let $f,h\in\mc{D}(M)$ thus
\begin{equation}
\label{12091633}
\begin{aligned}
\mathsterling_{\mr{grad}_{\mc{M}}(f)}(h)
&=
\sum_{i=1}^{n}\ep_{i}\lr{\mr{grad}_{\mc{M}}(f)}{E_{i}}_{\mc{M}}\cdot\mathsterling_{E_{i}}(h)
\\
&=
\sum_{i=1}^{n}\ep_{i}\cdot\mathsterling_{E_{i}}(f)\mathsterling_{E_{i}}(h)
\\
&=
\sum_{i=1}^{n}\ep_{i}\lr{\mr{grad}_{\mc{M}}(h)}{E_{i}}_{\mc{M}}\cdot\mathsterling_{E_{i}}(f)
\\
&=
\mathsterling_{\mr{grad}_{\mc{M}}(h)}(f);
\end{aligned}
\end{equation}
and \eqref{11301607a} follows.
Next let $W\in\mf{X}(M)$, so $\mathsterling_{W}$ is 
$(\tau_{c}^{\infty},\tau_{c}^{\infty})$-continuous thus
\begin{equation*}
\begin{aligned}
\exp_{M}^{U}(t)\circ\mathsterling_{W}
&=
\sum_{k=0}^{\infty}\frac{t^{k}}{k!}
\mathsterling_{U}^{k}
\mathsterling_{W},
\\
\mathsterling_{W}\circ\exp_{M}^{U}(t)
&=
\sum_{k=0}^{\infty}\frac{t^{k}}{k!}
\mathsterling_{W}
\mathsterling_{U}^{k};
\end{aligned}
\end{equation*}
both converging in $\mf{L}_{s}(\mc{D}(M))$.
Since $\mathsterling:\mf{X}(M)\to\mr{Der}(\mc{D}(M))$ is a Lie algebra isomorphism onto the Lie algebra of 
derivations of $\mc{D}(M)$, we have
$[W,U]=\ze
\Leftrightarrow
\mathsterling_{[W,U]}=\ze
\Leftrightarrow
[\mathsterling_{W},\mathsterling_{U}]=\ze
\Leftrightarrow
(\forall n\in\Z_{+})([\mathsterling_{W},\mathsterling_{U}^{n}]=\ze)$
therefore 
\begin{equation}
\label{12091629}
(\forall W\in\mf{X}(M))
([W,U]=\ze\Rightarrow[\mathsterling_{W},\exp_{M}^{U}(t)]=\ze).
\end{equation}
Next by the defining property of derivations, the second equality of \eqref{12091633} and Rmk. \ref{11240847} 
we deduce that 
\begin{equation}
\label{12091643}
[\mathsterling_{\mr{grad}_{\mc{M}}(f)},\mathsterling_{U}]
=
\sum_{i=1}^{n}\ep_{i}\cdot\mathsterling_{E_{i}}(f)\mathsterling_{[E_{i},U]}+
(-1)\sum_{i=1}^{n}\ep_{i}\cdot(\mathsterling_{U}\circ\mathsterling_{E_{i}})(f)\mathsterling_{E_{i}}.
\end{equation}
Let $A$ and $B$ denote the left and right sides of the equality \eqref{11301607b} respectively, thus
\begin{equation*}
\begin{aligned}
A&=\exp_{M}^{U}(t)\circ[\mathsterling_{\mr{grad}_{\mc{M}}(f)},\mathsterling_{U}];
\\
B&=[\mathsterling_{(\mr{grad}_{\mc{M}}\circ\exp_{M}^{U}(t))(f)},\mathsterling_{U}]\circ\exp_{M}^{U}(t).
\end{aligned}
\end{equation*}
Now if $[U,E_{i}]=\ze$ for every $i\in[1,n]\cap\Z$, then by \eqref{12091643} and Rmk. \ref{11240847} we obtain 
\begin{equation*}
\begin{aligned}
A&=
(-1)\sum_{i=1}^{n}\ep_{i}\cdot
(\exp_{M}^{U}(t)\circ\mathsterling_{U}\circ\mathsterling_{E_{i}})(f)(\exp_{M}^{U}(t)\circ\mathsterling_{E_{i}});
\\
B&=
(-1)\sum_{i=1}^{n}\ep_{i}\cdot
(\mathsterling_{U}\circ\mathsterling_{E_{i}}\circ\exp_{M}^{U}(t))(f)
(\mathsterling_{E_{i}}\circ\exp_{M}^{U}(t));
\end{aligned}
\end{equation*}
then st. \eqref{11301607b} follows since \eqref{12091629}.
\end{proof}
\begin{remark}
\label{12101810}
Since \eqref{12091629} we have 
$[U,W]=\ze\Rightarrow(\forall t\in\R)(\Lambda_{M}^{U}(t)(\mathsterling_{W})=\mathsterling_{W})$.
\end{remark}
\begin{lemma}
\label{12031345}
Let $\mc{M}=(M,g)$, $\mc{N}=(N,g^{\prime})$ be semi-Riemannian manifolds and $\phi:N\to M$ be a smooth 
diffeomorphism such that $\phi^{\ast}g=g^{\prime}$. Thus for every $f\in\mc{D}(M)$ we have 
\begin{equation*}
\phi^{\ast}\circ\mathsterling_{\mr{grad}_{\mc{M}}(f)}
=
\mathsterling_{(\mr{grad}_{\mc{N}}\circ\phi^{\ast})(f)}\circ\phi^{\ast}.
\end{equation*}
\end{lemma}
\begin{proof}
Let $\{E_{i}\}_{i=1}^{n=\mr{dim}M}$ be a frame of orthonormal fields of $\mc{M}$, and set
$G_{i}\coloneqq d(\phi^{-1})\circ E_{i}\circ\phi$, thus $E_{i}$ and $G_{i}$ are $\phi$-related, hence 
$\phi^{\ast}\circ\mathsterling_{E_{i}}=\mathsterling_{G_{i}}\circ\phi^{\ast}$ and 
$\{G_{i}\}_{i=1}^{n}$ is a frame of orthonormal fields of $\mc{N}$ since \cite[Lemma 1.6.6]{sil2cat}.
Therefore since again \cite[Lemma 1.6.6]{sil2cat} we have that 
\begin{equation*}
\begin{aligned}
\phi^{\ast}\circ\mathsterling_{\mr{grad}_{\mc{M}}(f)} 
&=
\sum_{i=1}^{n}
\phi^{\ast}(\lr{E_{i}}{E_{i}}_{\mc{M}})
\phi^{\ast}(\mathsterling_{E_{i}}(f))(\phi^{\ast}\circ\mathsterling_{E_{i}})
\\
&=
\sum_{i=1}^{n}
\lr{G_{i}}{G_{i}}_{\mc{N}}
\mathsterling_{G_{i}}(\phi^{\ast}f)
(\mathsterling_{G_{i}}\circ\phi^{\ast})
\\
&=
\mathsterling_{(\mr{grad}_{\mc{N}}\circ\phi^{\ast})(f)}
\circ
\phi^{\ast}.
\end{aligned}
\end{equation*}
\end{proof}
\begin{definition}
[The map $\mc{Z}$]
\label{12031517}
Let $\mc{M}$ be a semi-Riemannian manifold, $M$ be the manifold underlying $\mc{M}$, $U\in\mf{X}(M)$ 
and $X\subseteq M$. Define
\begin{equation*}
\begin{aligned}
\mc{Z}_{(\mc{M},U)}^{X}:\mc{D}_{1}(M,X)&\to\mf{B}(\mc{M},X)
\\
(f,\lambda)&\mapsto\mathsterling_{[\mr{grad}_{\mc{M}}(f),U]}+\lambda\un;
\end{aligned}
\end{equation*}
where $\un$ is the unit element of the unital algebra $\mf{B}(\mc{M})$.
\end{definition}
\begin{proposition}
\label{12041048}
Def. \ref{12031517} is well-set namely 
$\mc{Z}_{(\mc{M},U)}^{X}(f)\in\mf{B}(\mc{M},X)$ for every $f\in\mc{D}(M,X)$. 
\end{proposition}
\begin{proof}
$\mc{Z}_{(\mc{M},U)}^{X}(f)\in\mf{B}(\mc{M})$ since Rmk. \ref{11251447} and \eqref{12040058}.
Next by Notation we know that the elements in $\mr{DiffOp^{1}}(M)$ are local, and that 
$D\in\mr{DiffOp}^{1}(M)\Rightarrow D^{\dagger}\in\mr{DiffOp}^{1}(M)$,
thus $l\in\mc{D}(M,X)$ implies $\mr{supp}(\mc{Z}_{(\mc{M},U)}^{X}(f)(l))\subseteq\mr{supp}(l)\subseteq X$
and $\mr{supp}(\mc{Z}_{(\mc{M},U)}^{X}(f)^{\dagger}(l))\subseteq X$.
\end{proof}
\begin{corollary}
\label{12040944}
Let $\mc{M}$ be a semi-Riemannian manifold, $M$ be the manifold underlying $\mc{M}$ and $X\subseteq M$, thus
\begin{equation*}
\bigl(f\mapsto\mathsterling_{\mr{grad}_{\mc{M}}(f)}\bigr)\in\mf{L}(\mc{D}(M,X),\mf{B}(\mc{M},X)).
\end{equation*}
\end{corollary}
\begin{proof}
Let $\xi$ be the map in the statement, $\tilde{\xi}:\mc{D}(M)\to\mf{B}(\mc{M})$ be the map 
$f\mapsto\mathsterling_{\mr{grad}_{\mc{M}}(f)}$, and let 
$\underline{\xi}\coloneqq\imath_{\mf{B}(\mc{M})}^{\mf{L}(\mc{D}(M))}\circ\tilde{\xi}$.
The above maps are well-set namely 
$\xi(\mc{D}(M,X))\subseteq\mf{B}(\mc{M},X)$ since the proof of Prp. \ref{12041048} 
while $\tilde{\xi}(\mc{D}(M))\subseteq\mf{B}(\mc{M})$ since Rmk. \ref{11251447}.
We claim to show that
\begin{equation}
\label{10121245}
\tilde{\xi}\in\mf{L}(\mc{D}(M),\mf{B}(\mc{M}));
\end{equation}
which would prove our statement since 
$\imath_{\mf{B}(\mc{M},X)}^{\mf{B}(\mc{M})}\circ\xi=\tilde{\xi}\circ\imath_{\mc{D}(M,X)}^{\mc{D}(M)}$.
Now since Lemma \ref{11250550} we have that \eqref{10121245} would follow
if we prove that $\underline{\xi}\in\mf{L}(\mc{D}(M),\mf{L}_{b}(\mc{D}(M)))$.
But $\mc{D}(M)$ is barrelled therefore by applying the Banach-Steinhaus Thm. the above statement would follow 
if $\underline{\xi}\in\mf{L}(\mc{D}(M),\mf{L}_{s}(\mc{D}(M)))$ which at once follows since \eqref{11301607a}.
\end{proof}
For any semi-Riemannian manifold $\mc{M}$ let $\mr{Onf}(\mc{M})$ denote the set of frames of orthonormal fields 
of $\mc{M}$.
\begin{definition}
\label{12041107}
Let $\mf{Vf}$ be the unique full subcategory of $\mf{vf}$ such that 
\begin{equation*}
\mr{Obj}(\mf{Vf})=\bigl\{(\mc{M},U)\in\mr{Obj}(\mf{vf})\,\big\vert\,
(\exists\{E_{i}\}_{i=1}^{n=\mr{dim}M}\in\mr{Onf}(\mc{M}))
(\forall i\in[1,n]\cap\Z)([U,E_{i}]=\ze)\bigr\}.
\end{equation*}
\end{definition}
\begin{definition}
\label{12041131}
Define $\ms{x}\coloneqq\ps{\Uppsi}_{0}\circ\mf{x}\circ\mr{I}_{\mf{Vf}}^{\mf{vf}}$ and 
$\ms{z}\coloneqq\ps{\Uppsi}_{0}\circ\mf{z}\circ\mr{I}_{\mf{Vf}}^{\mf{vf}}$.
\end{definition}
\begin{corollary}
[The $0-$species $\ms{x}$ and $\ms{z}$]
$\ms{x}\in\mf{Sp}_{0}(\mf{Vf})$ and $\ms{z}\in\mf{Sp}_{0}(\mf{Vf})$. 
\end{corollary}
\begin{proof}
Since Thm. \ref{11280957} and Thm. \ref{11290640}.
\end{proof}
\begin{theorem}
[Main. A natural transformation from $\ms{x}$ to $\ms{z}$]
\label{12031824}
There exists a natural transformation
\begin{equation*}
\mf{J}\in\mr{Mor}_{\mr{Fct}(\mf{Vf},\mf{Chdv}_{0})}(\ms{x},\ms{z})
\end{equation*}
uniquely determined by the following properties:
\begin{equation*}
\begin{aligned}
\mf{J}:\mr{Obj}(\mf{Vf})\ni(\mc{M},U)
&\mapsto
\mf{J}(\mc{M},U)
=
(\un_{\lr{\mc{M}}{U}},\mc{J}_{(\mc{M},U)}^{\dagger},\mc{J}_{(\mc{M},U)});
\\
\mc{J}_{(\mc{M},U)}:\mr{Obj}(\lr{\mc{M}}{U})
&\ni
X\mapsto\mc{Z}_{(\mc{M},U)}^{X};
\\
\mc{J}_{(\mc{M},U)}^{\dagger}:\mr{Obj}(\lr{\mc{M}}{U})
&\ni
X
\mapsto(\mc{Z}_{(\mc{M},U)}^{X})^{\dagger}.
\end{aligned}
\end{equation*}
\end{theorem}
\begin{proof}
Let us take the properties of $\mf{J}$ given in the statement as its definition, then show that 
the definition is well-set and that $\mf{J}$ is a natural transformation from $\ms{x}$ to $\ms{z}$.
Let $(\mc{M},U)\in\mr{Obj}(\mf{Vf})$ and $X,Y\in\lr{\mc{M}}{U}$.
By $\mathsterling_{[\mr{grad}_{\mc{M}}(f),U]}=[\mathsterling_{\mr{grad_{\mc{M}}}(f)},\mathsterling_{U}]$
and since $\mf{B}(\mc{M},X)$ is a topological algebra we obtain by Cor. \ref{12040944} that 
\begin{equation}
\label{12041155}
\bigl(f\mapsto\mathsterling_{[\mr{grad}_{\mc{M}}(f),U]}\bigr)\in\mf{L}(\mc{D}(M,X),\mf{B}(\mc{M},X)).
\end{equation}
Since \eqref{12041155} and \eqref{12040058} we have
\begin{equation}
\label{12040740}
\mc{J}_{(\mc{M},U)}(X)\in\mr{Mor}_{\mr{tls}}(\mf{q}_{0}(\mc{D}_{1}(M,X)),\mf{q}_{0}(\mf{B}(\mc{M},X))).
\end{equation}
Next \eqref{11301607b} is equivalent to say that for all $f\in\mc{D}(M)$ we have
\begin{equation*}
\Gamma_{\mc{M}}^{U}(t)
(\mathsterling_{[\mr{grad}_{\mc{M}}(f),U]})
=
\mathsterling_{[(\mr{grad}_{\mc{M}}\circ\exp_{M}^{U}(t))(f),U]};
\end{equation*}
therefore by considering \eqref{12040740} we have that 
the following is a commutative diagram in the category $\mr{tls}$
\begin{equation*}
\xymatrix{
\mf{q}_{0}(\mc{D}_{1}(M,Y))
\ar[rr]^{\mc{J}_{(\mc{M},U)}(Y)}
&&
\mf{q}_{0}(\mf{B}(\mc{M},Y))
\\
&&
\\
\mf{q}_{0}(\mc{D}_{1}(M,X))
\ar[rr]_{\mc{J}_{(\mc{M},U)}(X)}
\ar[uu]^{\mf{q}_{0}(\exp_{M}^{U}(t)\oplus\mr{Id}_{\C})}
&&
\mf{q}_{0}(\mf{B}(\mc{M},X))
\ar[uu]_{\mf{q}_{0}(\Gamma_{\mc{M}}^{U}(t))}}
\end{equation*}
namely
\begin{equation}
\label{12040741}
\mc{J}_{(\mc{M},U)}\in\mr{Mor}_{\mr{Fct}(\lr{\mc{M}}{U},\mr{tls})}
(\mf{q}_{0}\circ\mr{F}_{\lr{\mc{M}}{U}},\mf{q}_{0}\circ\mc{F}_{\lr{\mc{M}}{U}}).
\end{equation}
Thus by Def. \ref{12051824} and the corresponding of \cite[Cor. 1.4.16]{sil2cat}
\begin{equation}
\label{12040742}  
\mc{J}_{(\mc{M},U)}^{\dagger}\in
\mr{Mor}_{\mr{Fct}(\lr{\mc{M}}{U}^{\mr{op}},\mr{tls})}
(\mf{r}\circ\mc{F}_{\lr{\mc{M}}{U}}^{\dagger},\mf{r}\circ\mr{F}_{\lr{\mc{M}}{U}}^{\dagger}).
\end{equation}
\eqref{12040741} and \eqref{12040742} imply 
\begin{equation}
\label{12040743}
\mf{J}(\mc{M},U)\in\mr{Mor}_{\mf{Chdv}_{0}}(\ms{x}(\mc{M},U),\ms{z}(\mc{M},U)).
\end{equation}
Therefore the statement will follow if we prove that for every $(\mc{N},V)\in\mr{Obj}(\mf{Vf})$ and 
every $\phi\in\mr{Mor}_{\mf{Vf}}((\mc{M},U),(\mc{N},V))$ the following is a commutative diagram in 
the category $\mf{Chdv}_{0}$ 
\begin{equation}
\label{12041740}
\xymatrix{
\ms{x}(\mc{N},V)
\ar[rr]^{\mf{J}(\mc{N},V)}
&&
\ms{z}(\mc{N},V)
\\
&&
\\
\ms{x}(\mc{M},U)
\ar[rr]^{\mf{J}(\mc{M},U)}
\ar[uu]^{\ps{\Uppsi}_{0}(f^{\phi},\ms{T}_{1}^{\phi})}
&&
\ms{z}(\mc{M},U)
\ar[uu]_{\ps{\Uppsi}_{0}(f^{\phi},\mf{T}^{\phi})}}
\end{equation}
namely 
\begin{equation*}
\bigl(f^{\phi},(\un_{\mf{q}_{0}}\ast\mf{T}^{\phi})^{\dagger},\un_{\mf{q}_{0}}\ast\mf{T}^{\phi}\bigr)
\circ
\bigl(\un_{\lr{\mc{M}}{U}},\mc{J}_{(\mc{M},U)}^{\dagger},\mc{J}_{(\mc{M},U)}\bigr)
=
\bigl(\un_{\lr{\mc{N}}{V}},\mc{J}_{(\mc{N},V)}^{\dagger},\mc{J}_{(\mc{N},V)}\bigr)
\circ
\bigl(f^{\phi},(\un_{\mf{q}_{0}}\ast\ms{T}_{1}^{\phi})^{\dagger},\un_{\mf{q}_{0}}\ast\ms{T}_{1}^{\phi}\bigr);
\end{equation*}
that according to \eqref{123121130} is equivalent to 
\begin{multline*}
\bigl(\un_{\lr{\mc{M}}{U}}\circ f^{\phi},
(\mc{J}_{(\mc{M},U)}^{\dagger}\ast\un_{f^{\phi}})\circ(\un_{\mf{q}_{0}}\ast\mf{T}^{\phi})^{\dagger},
(\un_{\mf{q}_{0}}\ast\mf{T}^{\phi})\circ(\mc{J}_{(\mc{M},U)}\ast\un_{f^{\phi}})
\bigr)
=\\
\bigl(f^{\phi}\circ\un_{\lr{\mc{N}}{V}},
((\un_{\mf{q}_{0}}\ast\ms{T}_{1}^{\phi})^{\dagger}
\ast\un_{\un_{\lr{\mc{N}}{V}}})\circ\mc{J}_{(\mc{N},V)}^{\dagger},
\mc{J}_{(\mc{N},V)}\circ((\un_{\mf{q}_{0}}\ast\ms{T}_{1}^{\phi})\ast\un_{\un_{\lr{\mc{N}}{V}}})
\bigr).
\end{multline*}
which reduces to the following equality of morphisms of the category $\mr{Fct}(\lr{\mc{N}}{V},\mr{tls})$
\begin{equation*}
(\un_{\mf{q}_{0}}\ast\mf{T}^{\phi})\circ(\mc{J}_{(\mc{M},U)}\ast\un_{f^{\phi}})
=
\mc{J}_{(\mc{N},V)}\circ((\un_{\mf{q}_{0}}\ast\ms{T}_{1}^{\phi})\ast\un_{\un_{\lr{\mc{N}}{V}}}).
\end{equation*}
Which is equivalent to say that for every $Z\in\lr{\mc{N}}{V}$ we have the following equality of morphisms 
of the category $\mr{tls}$
\begin{equation*}
\mf{T}^{\phi}(Z)\circ\mc{J}_{(\mc{M},U)}(\phi(Z))
=
\mc{J}_{(\mc{N},V)}(Z)\circ\ms{T}_{1}^{\phi}(Z);
\end{equation*}
namely
\begin{equation*}
\mc{T}(\phi)\circ\mc{Z}_{(\mc{M},U)}^{\phi(Z)}
=
\mc{Z}_{(\mc{N},V)}^{Z}\circ(\phi^{\ast}\oplus\mr{Id}_{\C});
\end{equation*}
equivalent to say that for every $(f,\lambda)\in\mc{D}_{1}(M,\phi(Z))$ we have 
\begin{equation*}
\phi^{\ast}
(\mathsterling_{[\mr{grad}_{\mc{M}}(f),U]}+\lambda\un)
(\phi^{-1})^{\ast}
=
\mathsterling_{[(\mr{grad}_{\mc{N}}\circ\phi^{\ast})(f),V]}+\lambda\un;
\end{equation*}
namely
\begin{equation*}
\phi^{\ast}\mathsterling_{[\mr{grad}_{\mc{M}}(f),U]}(\phi^{-1})^{\ast}
=
\mathsterling_{[(\mr{grad}_{\mc{N}}\circ\phi^{\ast})(f),V]};
\end{equation*}
but $\mathsterling$ is a Lie algebra morphism, so the above is equivalent to 
\begin{equation*}
[\phi^{\ast}\mathsterling_{\mr{grad}_{\mc{M}}(f)}(\phi^{-1})^{\ast},
\phi^{\ast}\mathsterling_{U}(\phi^{-1})^{\ast}]
=
[\mathsterling_{(\mr{grad}_{\mc{N}}\circ\phi^{\ast})(f)},\mathsterling_{V}];
\end{equation*}
which follows since Lemma \ref{12031345} and the fact that $U$ and $V$ are $\phi$-related.
Thus the diagram \eqref{12041740} is commutative and the statement follows.
\end{proof}
\section{Additional equicontinuity requests}
\label{12120940}
\begin{definition}
\label{12011752}
Let $M$ be a manifold and $U\in\mf{X}(M)$ such that 
$\{(\mathsterling_{U}^{\divideontimes})^{k}\,\vert\,k\in\Z_{+}\}$ is 
$(\tau^{\infty},\tau^{\infty})$-equicontinuous, define 
\begin{equation*}
\mr{Exp}_{M}^{U}\coloneqq\exp_{\mc{C}^{\infty}(M)}^{\mathsterling_{U}^{\divideontimes}}.
\end{equation*}
\end{definition}
Under the same reasoning used in Rmk. \ref{11240847} we have that
$\mr{Exp}_{M}^{U}$ is a group of unit preserving $\ast$-automorphisms of $\mc{C}^{\infty}(M)$. 
\begin{lemma}
\label{11301957}
Let $M$ be a manifold and $U\in\mf{X}(M)$ such that 
$\{(\mathsterling_{U}^{\divideontimes})^{k}\,\vert\,k\in\Z_{+}\}$ is 
$(\tau^{\infty},\tau^{\infty})$-equicontinuous
and 
$\{\mathsterling_{U}^{k}\,\vert\,k\in\Z_{+}\}$ is $(\tau^{\infty}\up\mc{D}(M),\tau_{c}^{\infty})$-equicontinuous.
Thus
\begin{equation*}
\mr{Exp}_{M}^{U}(t)\circ\imath_{\mc{D}(M)}^{\mc{C}^{\infty}(M)}
=
\imath_{\mc{D}(M)}^{\mc{C}^{\infty}(M)}\circ\exp_{M}^{U}(t);
\end{equation*}
in particular $\exp_{M}^{U}(t)\in\mf{L}(\mc{D}(M)[\tau^{\infty}])$. 
\end{lemma}
\begin{proof}
By the hypothesis and since $\mc{D}(M)\hookrightarrow\mc{C}^{\infty}(M)$ we have that
$\{\mathsterling_{U}^{k}\,\vert\,k\in\Z_{+}\}$ is $(\tau_{c}^{\infty},\tau_{c}^{\infty})$-equicontinuous.
Thus $\exp_{M}^{U}(t)$ is well-set and for every $f\in\mc{D}(M)$ we have 
\begin{equation*}
\begin{aligned}
\exp_{M}^{U}(t)f
&=
\sum_{k=0}^{\infty}\frac{t^{k}}{k!}\mathsterling_{U}^{k}(f)
\text{ w.r.t. the $\tau_{c}^{\infty}$-topology} 
\\
&=
\sum_{k=0}^{\infty}\frac{t^{k}}{k!}(\mathsterling_{U}^{\divideontimes})^{k}(f)
\text{ w.r.t. the $\tau^{\infty}$-topology}
\\
&=\mr{Exp}_{M}^{U}(t)f.
\end{aligned}
\end{equation*}
\end{proof}
\begin{proposition}
\label{12010855}
Under the hypothesis of Lemma \ref{11301957}, for every $t\in\R$, $F\in\mc{C}^{\infty}(M)$ and $f\in\mc{D}(M)$
we have $\exp_{M}^{U}(t)(F\cdot f)=\mr{Exp}_{M}^{U}(t)(F)\cdot\exp_{M}^{U}(t)(f)$.
\end{proposition}
\begin{proof}
Since Lemma \ref{11301957} and since $\mr{Exp}_{M}^{U}(t)$ is a $\ast$-automorphism of $\mc{C}^{\infty}(M)$.
\end{proof}
\begin{proposition}
\label{12001100}
Under the hypothesis of Lemma \ref{11301957}, 
let $\mc{A}\subset\mf{L}(\mc{D}(M))$ such that $\mc{A}$ is naturally a left $\mc{C}^{\infty}(M)$-module 
namely w.r.t. the external product defined by 
$(F\cdot Q):\mc{D}(M)\to\mc{D}(M)$, $h\mapsto F\cdot Q(h)$ for every $Q\in\mc{A}$ and $F\in\mc{C}^{\infty}(M)$. 
Thus for every $t\in\R$ we have
\begin{enumerate}
\item
if $h\in\mc{D}(M)$, then 
$\Lambda_{M}^{U}(t)(F\cdot Q)h=\mr{Exp}_{M}^{U}(t)(F)\cdot\Lambda_{M}^{U}(t)(Q)h$;
\label{12001100st1}
\item
if $\Lambda_{M}^{U}(t)\mc{A}\subseteq\mc{A}$, then 
$\Lambda_{M}^{U}(t)(F\cdot Q)=\mr{Exp}_{M}^{U}(t)(F)\cdot\Lambda_{M}^{U}(t)(Q)$ 
with the $(\cdot)$ operation in $\mc{A}$.
\label{12001100st2}
\end{enumerate}
\end{proposition}
\begin{proof}
St. \eqref{12001100st2} follows since st. \eqref{12001100st1}. Next we have 
\begin{equation*}
\begin{aligned}
\Lambda_{M}^{U}(t)(F\cdot Q)h
&=
\exp_{M}^{U}(t)\bigl(F\cdot(Q\exp_{M}^{U}(-t)h)\bigr)
\\
&=
\mr{Exp}_{M}^{U}(t)(F)\cdot\Lambda_{M}^{U}(t)(Q)h;
\end{aligned}
\end{equation*}
where the second equality follows by Prp. \ref{12010855}. 
\end{proof}
\begin{corollary}
\label{12021121}
Let $\mc{M}$ be a semi-Riemannian manifold, $M$ its underlying manifold and $\{E_{i}\}_{i}^{n=\mr{dim}M}$
be a frame of orthonormal fields of $\mc{M}$. Thus under the hypothesis of Lemma \ref{11301957}
we have for every $V\in\mf{X}(M)$, $t\in\R$ and $h\in\mc{D}(M)$ that 
\begin{equation}
\label{12021121st1}
\Gamma_{\mc{M}}^{U}(t)(\mathsterling_{V})h
=
\sum_{i=1}^{n}\ep_{i}\mr{Exp}_{M}^{U}(t)(\lr{V}{E_{i}}_{\mc{M}})\cdot
\Gamma_{M}^{U}(t)(\mathsterling_{E_{i}})h.
\end{equation}
If in addition $[U,E_{i}]=\ze$ for every $i\in[1,n]\cap\Z$, then 
\begin{equation}
\label{12021121st2}
\Gamma_{\mc{M}}^{U}(t)(\mathsterling_{V})
=
\mathsterling_{V_{t}}
\end{equation}
where $V_{t}\coloneqq\sum_{i=1}^{n}\ep_{i}\mr{Exp}_{M}^{U}(t)(\lr{V}{E_{i}}_{\mc{M}})E_{i}$.
Moreover the left $\mc{C}^{\infty}(M)$-submodule $\mc{A}_{M}$ of $\mr{DiffOp}^{1}(M)$ generated by the set 
$\{\mathsterling_{W}\,\vert\,W\in\mf{X}(M)\}$ is such that $\Gamma_{\mc{M}}^{U}(t)\mc{A}_{M}\subseteq\mc{A}_{M}$,
and 
\begin{equation}
\label{12021121st3}
\Gamma_{\mc{M}}^{U}(t)(\mathsterling_{V})
=
\sum_{i=1}^{n}\ep_{i}\mr{Exp}_{M}^{U}(t)(\lr{V}{E_{i}}_{\mc{M}})\cdot\mathsterling_{E_{i}};
\end{equation}
with the $(\cdot)$ operation in $\mc{A}_{M}$.
\end{corollary}
\begin{proof}
\eqref{12021121st1} follows since Prp. \ref{12001100}\eqref{12001100st1} applied to the natural
left $\mc{C}(M)^{\infty}$-module $\mr{DiffOp}^{1}(M)$.
Next if $[U,E_{i}]=\ze$ for every $i\in[1,n]\cap\Z$, then 
by \eqref{12021121st1} and Rmk. \ref{12101810} we obtain 
\begin{equation}
\label{120021331}
\begin{aligned}
\Gamma_{\mc{M}}^{U}(t)(\mathsterling_{V})h
&=
\sum_{i=1}^{n}\ep_{i}\mr{Exp}_{M}^{U}(t)(\lr{V}{E_{i}}_{\mc{M}})\cdot
\mathsterling_{E_{i}}h
\\
&=
\mathsterling_{V_{t}}h.
\end{aligned}
\end{equation}
Hence \eqref{12021121st2} follows, which together Prp. \ref{12001100}\eqref{12001100st1}
imply $\Gamma_{\mc{M}}^{U}(t)\mc{A}_{M}\subseteq\mc{A}_{M}$.
\eqref{12021121st3} follows since the first equality in \eqref{120021331} and the fact that 
$\mc{A}_{M}$ is a left $\mc{C}(M)^{\infty}$-module; alternatively by 
$\Gamma_{\mc{M}}^{U}(t)\mc{A}_{M}\subseteq\mc{A}_{M}$ and Prp. \ref{12001100}\eqref{12001100st2}.
\end{proof}
We conclude this section with a result providing sufficient conditions on a complete vector field $U$ on $\mc{M}$
ensuring that $\exp_{M}^{U}=\upeta_{M}^{U}$. 
\begin{corollary}
\label{12061539}
Let $(\mc{M},U)\in\mr{vf}^{\star}$ with $\mc{M}=(M,g)$ such that $U$ is complete and 
\begin{equation}
\label{12061540}
(\forall t\in\R)(\mu_{g}\circ\upeta_{M}^{U}(t)=\mu_{g}\circ\imath_{\mc{D}(M)}^{\mc{K}(M)}).
\end{equation}
Thus 
\begin{enumerate}
\item
$(\mc{M},U)\in\mf{vf}_{0}$ and $\upeta_{M}^{U}$ extends to a $C_{0}$-group $\upeta_{\mc{M}}^{U}$ on $\mc{H}_{g}$ 
of unitary operators whose infinitesimal generator extends $\mathsterling_{U}$.
\label{12061539st1}
\item
If $\{\mathsterling_{U}^{k}\,\vert\,k\in\Z_{+}\}$ is 
$(\|\cdot\|_{\mc{H}_{g}},\|\cdot\|_{\mc{H}_{g}})$-equicontinuous, then
$\mc{D}(M)$ is a core for both the infinitesimal generators of $\upeta_{\mc{M}}^{U}$ and $\mf{exp}_{\mc{M}}^{U}$.
Thus $\mf{exp}_{\mc{M}}^{U}=\upeta_{\mc{M}}^{U}$, in particular $\exp_{M}^{U}=\upeta_{M}^{U}$. 
\label{12061539st2}
\end{enumerate}
\end{corollary}
\begin{proof}
Let $t\in\R$. $\upeta_{M}^{U}(t)$ is an isometry of the $\mc{H}_{g}$-dense subspace $\mc{D}(M)$
since $\upeta_{\mc{M}}^{U}$ is a morphism of $\ast$-algebras and since \eqref{12061540},
thus $\upeta_{M}^{U}(t)$ extends to a unitary operator $\upeta_{\mc{M}}^{U}(t)$ on $\mc{H}_{g}$. 
Next let $\{t_{n}\}_{n\in\N}$ be a sequence in $\R$ converging at $0$ and $f\in\mc{D}(M)$, thus 
$\lim_{n\in\N}\upeta_{\mc{M}}^{U}(t_{n})(f)=f$ pointwise since the flow $\uptheta^{U}$ is pointwise continuous.
But $\upeta_{\mc{M}}^{U}(t_{n})(f)\in\mc{D}(M)$ as well $f\in\mc{D}(M)$ and $\mc{D}(M)\subset\mc{H}_{g}$
therefore by applying the Lebesgue Thm.
$\lim_{n\in\N}\upeta_{\mc{M}}^{U}(t_{n})(f)=f$ w.r.t. the norm topology of $\mc{H}_{g}$.
But $\upeta_{\mc{M}}^{U}$ is a semigroup of norm continuous operators, therefore we conclude that 
$\upeta_{\mc{M}}^{U}$ is a $C_{0}$-group on $\mc{H}_{g}$ and we let $\mf{m}_{U}$ denote its infinitesimal 
generator.
Next by definition of $\uptheta^{U}$ we deduce that for every $f\in\mc{D}(M)$,
$\mathsterling_{U}f$ is the pointwise derivative at $t=0$ of the map $t\mapsto\upeta_{\mc{M}}^{U}(t)f$.
Thus by $\mc{D}(M)\subset\mc{H}_{g}$ and by applying the Lebesgue Thm. 
we conclude that $\mathsterling_{U}f$ is the derivative at $t=0$ of the map $t\mapsto\upeta_{\mc{M}}^{U}(t)f$
w.r.t. the norm topology of $\mc{H}_{g}$, namely $\mf{m}_{U}$ is an extension of $\mathsterling_{U}$.  
Now $\upeta_{\mc{M}}^{U}(t)f\in\mr{L}^{1}(M,d\mu_{g})$ 
since $\mu_{g}(|\upeta_{\mc{M}}^{U}(t)f|)=\mu_{g}(\upeta_{\mc{M}}^{U}(t)|f|)=\mu_{g}(|f|)<\infty$ 
where the first equality follows at once by the definition of $\upeta_{M}^{U}$, the second equality follows 
by \eqref{12061540}, the inequality follows since $\mc{D}(M)\subset\mc{K}(M)\subset\mr{L}^{1}(M,d\mu_{g})$.
Moreover $\mathsterling_{U}f\in\mc{D}(M)\subset\mr{L}^{1}(M,d\mu_{g})$ thus by applying the Lebesgue Thm. 
similarly as above, we obtain that $\mathsterling_{U}f$ is the derivative at $t=0$ of the map 
$t\mapsto\upeta_{\mc{M}}^{U}(t)f$. w.r.t. the norm topology of $\mr{L}^{1}(M,d\mu_{g})$. 
Next $\mu_{g}$ extends to an element of $\mr{L}^{1}(M,\mu_{g})^{\prime}$, therefore what right now proven
and \eqref{12061540} imply that $\mu_{g}\circ\mathsterling_{U}=\ze$ and st. \eqref{12061539st1} follows.
Now st. \eqref{12061539st1} and Cor. \ref{11261248} imply that there exists a unique 
$C_{0}$-group $\mf{exp}_{\mc{M}}^{U}$ on $\mc{H}_{g}$ of unitary operators extending $\exp_{M}^{U}$ 
and whose infinitesimal generator $\mf{l}_{U}$ extends $\mathsterling_{U}$, in particular 
\begin{equation}
\label{12070638}
\mf{m}_{U}\up\mc{D}(M)=\mf{l}_{U}\up\mc{D}(M)=\mathsterling_{U}.
\end{equation}
Therefore the additional equicontinuity hypothesis in st. \eqref{12061539st2} 
implies that $\mc{D}(M)$ is a set of analytic elements for both the generators 
$\mf{m}_{U}$ and $\mf{l}_{U}$, moreover $\mc{D}(M)$ is dense in $\mc{H}_{g}$ w.r.t. the norm topology thus
also w.r.t. the $\sigma(\mc{H}_{g},\mc{H}_{g}^{\prime})$-topology, 
and $\mathsterling_{U}\mc{D}(M)\subseteq\mc{D}(M)$, thus we conclude that $\mc{D}(M)$ is a core of 
$\mf{m}_{U}$ and $\mf{l}_{U}$ by applying well-known general results about the core of generators of 
$C_{0}$-semigroups, and then the first sentence of st. \eqref{12061539st2} 
follows. The first sentence of st. \eqref{12061539st2} and \eqref{12070638} imply $\mf{m}_{U}=\mf{l}_{U}$ 
and then the second sentence of st. \eqref{12061539st2} follows by the well-known uniqueness of 
the generator of a equicontinuous $C_{0}$-semigroup.
\end{proof}


\bibliography{references}
\bibliographystyle{amsplain}

\end{document}